\documentclass [11pt,oneside] {amsart}
\usepackage[leqno]{amsmath}
\usepackage{amsthm}
\usepackage{dsfont}
\usepackage{amsfonts}
\usepackage{amssymb}
\usepackage{eucal}
\usepackage{upgreek}
\usepackage{latexsym,amscd,fullpage, tikz, cmll, mathrsfs, epic, mathtools} 
\usepackage{extpfeil}
\usepackage{color}
\usepackage{graphicx}
\usepackage [all]{xy}
\usetikzlibrary{fit}
\usetikzlibrary{arrows}
\usetikzlibrary{matrix}
\SelectTips{cm}{12}

\theoremstyle{definition}
  \newtheorem{theorem}{Theorem}[subsection]
  
  \newtheorem{lemma}[theorem]{Lemma}

\theoremstyle{remark} 
  \newtheorem{remark}[theorem]{Remark}
  \newtheorem{example}[theorem]{Example}

\theoremstyle{definition}
  \newtheorem{definition}[theorem]{Definition}

\newcommand {\Aut}{\operatorname{Aut}}
\newcommand {\Hom} {\operatorname {Hom}}

\renewcommand {\geq} {\geqslant}

\newcommand {\sub} {\subseteq}
\newcommand{\FM}[1]{\mathcal{F\!M}_{#1\!}}
\newcommand{\con}[1]{C_{#1}\ \!\!}
\newcommand{\non}[1]{\mathrm{#1}}
\newcommand{\isop}[1]{{#1}_{\non{iso}}^{\non{op}}}
\newcommand{\Sets}{\bold{Sets}}
\newcommand{\fnS}{\bold{S}}

\newcommand{\FVect}[1]{\textsf{F}\bold{Vect}_{#1}}
\newcommand{\Sch}[1]{\bold{Sch}_{#1}}
\newcommand{\Mon}[2]{\bold{Mon}_{#1}{#2}}

\newcommand {\tensor}{\otimes}
\newcommand{\Spec}[1]{\mathrm{Spec}\ \!#1}
\newcommand{\ZZ}{\mathbb{Z}}
\newcommand{\qr}[1]{\underline{\mathrm{#1}}}
\newcommand{\conqr}[1]{\qr{C}_{#1}}

\def\Ac{\mathcal{A}}
\def\Bc{\mathcal{B}}
\def\Cc{\mathcal{C}}

\def\Dc{\mathcal{D}}

\def\Mc{\mathcal{M}}

\def\Qc{\mathcal{Q}}

\def\1{{\bf 1}}
\def\lra{\longrightarrow}

\DeclareSymbolFont{largesym}{OML}{cmm}{m}{it}
\DeclareMathSymbol{\nstnsmall}{0}{largesym}{"22}
\newcommand{\btd}{\mathrel{\scalebox{1.75}{\color{lightgray}{$\blacktriangledown$}}}}

\title{ Higher configuration operads\\{\small by way of}\\quiver Grassmannians} 
\author{Tyler Foster}
\address{Department of Mathematics\\
Yale University\\
New Haven, CT 06511, USA}
\email{tyler.foster@yale.edu}
\date {\today}
\thanks{The author would like to thank, above all, Mikhail Kapranov for guidance at every stage of this work. Many thanks go to Gennaro di Brino, Tim Kramer, Sam Payne, and Markus Reineke for helpful conversations. The author would also like to acknowledge conversations with Tobias Dyckerhoff and Jacob Lurie that were of great help at early stages in this work.}

\begin{document}
\maketitle

\begin{abstract}
We introduce a construction that associates, to each finite dimensional $\Bbbk$-vector space $V$, a $\mathbb{Z}_{\geq0}$-indexed family $\mathrm{FM}_{V}=\big\{\mathrm{FM}_{V\!}(n)\big\}_{n=1}^{\infty}$ of projective $\Bbbk$-varieties. This family comes naturally equipped with the structure of a operad in the category of $\Bbbk$-schemes. When $\mathrm{dim}_{\!\Bbbk} V=1$, the operad $\mathrm{FM}_{V}$ contains, as a suboperad, the family $\big\{\overline{M}_{0,n+1}\big\}_{n=2}^{\infty}$ formed by the moduli spaces of stably marked rational curves. For $n=\mathrm{dim}_{\Bbbk} V$ arbitrary, $\mathrm{FM}_{V}$ contains the family of Chen-Gibney-Krashen moduli spaces of stably marked trees of projective $n$-spaces as a suboperad.

We realize the operad $\mathrm{FM}_{V}$ as part of a larger theory that describes how to construct operads from suitable functors. Given a category $\Cc$ that satisfies conditions allowing us to consider the operation of "substituting a value into an argument" within $\Cc$, and given any functor $F:\Cc\lra\Dc$ satisfying a variant of right-exactness that expresses preservation of substitutions, we use $F$ to construct a version of a set-valued operad whose inputs are given by objects in $\Cc$. When $\Cc$ happens to be the category of nonempty finite sets, we obtain operads in the classic sense.

Finally, we show that in many cases it is possible to realize these operads as disjoint unions of quiver Grassmannians, the family $\mathrm{FM}_{V}=\big\{\mathrm{FM}_{V\!}(n)\big\}_{n=1}^{\infty}$ being one example.
\end{abstract}

\vskip 1cm

\section*{Introduction}

Our overarching concern in this text is the study of operads of algebraic varieties over a fixed ground field $\Bbbk$. The family
$$
\big\{\overline{M}_{0,n+1}\big\}_{n=1}^{\infty}
$$
formed by the moduli spaces $\overline{M}_{0,n+1}$ of stable $(n+1)$-marked rational curves, where by definition $\overline{M}_{0,1+1}:=\Spec{\mathbb{C}}$, is perhaps the best known example of such an operad. We denote this operad $\mathrm{M}$. More recent examples appear in \cite{CGK}. In \cite{CGK}, Chen, Gibney, and Krashen define, for each pair of positive integers $d$ and $n$, a smooth algebraic $\Bbbk$-variety that we will here denote
$$
\mathrm{CGK}_{d}(n)
$$
The variety $\mathrm{CGK}_{d}(n)$ is a moduli space classifying geometric objects called {\em stable} $n$-{\em marked trees of projective} $d$-{\em space}. This example will be of central importance to us, so we take a moment to discuss it in more detail:

\subsection{Chen-Gibney-Krashen spaces}
For each nonnegative integer $N$, let a {\em projective $N$-space over} $\Bbbk$ to be any $\Bbbk$-variety $P$ isomorphic to $\mathbb{P}^N_{\Bbbk}$. Fix a vector space $V$, let $d=\mathrm{dim}_{\Bbbk}V$, and define $H$ to be the projective $(d-1)$-space $H=\mathbb{P}(V)$, the projectivization of $V$. For each integer $n\geq2$, define an $n$-{\em marked projective} $d$-{\em space} to be any triple
$$
(P,\ \iota,\ p)
$$
consisting of a projective $d$-space $P$, a closed embedding $\iota:H\hookrightarrow P$, and an injection
$$
p_{(-)}:\{1,\dots,n\}\xhookrightarrow{\ \ \ }P(\Bbbk)-H(\Bbbk)
$$
of sets, whose values $p_{i}$ we call {\em marked points in} $P$.

If $\mathscr{B}\ell_{p_i}P$ denotes the blow-up of $P$ at one of these marked points $p_{i}\in P(\Bbbk)-H(\Bbbk)$, then the projective geometry of $P$ induces a canonical isomorphism between the $(n-1)$-dimensional variety underlying the exceptional divisor $D$ in $\mathscr{B}\ell_{p_i}P$ and the subvariety $\iota(H)\subset P$. Thus, given a second integer $n'\geq2$, and an $n'$-marked projective $d$-space $(P',\iota',p')$, there is a canonical way to glue the projective $d$-space $P'$ to the blow-up $\mathscr{B}\ell_{p_i}P$, via the identification of $\iota'(H)\in P'$ with the exceptional divisor $D$ in $\mathscr{B}\ell_{p_i}P$. Upon gluing, we obtain a new variety
$$
B\ \ =\ \ P'\underset{\!\!\!\!\iota'(H)\sim D\!}{\sqcup}\mathscr{B}\ell_{p_i}P
$$
that comes equipped with a closed embedding $\jmath:H\hookrightarrow B$, induced by our embedding $\iota:H\hookrightarrow P$, and with an injection $q:\{1,\dots,n+n'-1\}\xhookrightarrow{\ \ \ }B(\Bbbk)-H(\Bbbk)$ of sets.

Inductively then, we define a {\em stable} $n$-{\em tree of projective} $d$-{\em spaces} to be any triple $(B,\jmath,q)$ consisting of a $\Bbbk$-variety $B$, a closed embedding $\jmath:H\hookrightarrow B$, and an inclusion $q:\{1,\dots,n\}\xhookrightarrow{\ \ \ }B(\Bbbk)-H(\Bbbk)$, such that that we can obtain this triple by iterating the above construction.
An {\em isomorphism} $B\xrightarrow{\ \sim\ }B'$ of stable $n$-trees of projective $d$-spaces is any isomorphism of the underlying varieties $B$ and $B'$, which commutes with the embeddings $H\hookrightarrow B$ and $H\hookrightarrow B'$, and commutes with the markings $\{1,\dots,n\}\xhookrightarrow{\ \ \ }B(\Bbbk)-H(\Bbbk)$ and $\{1,\dots,n\}\xhookrightarrow{\ \ \ }B'(\Bbbk)-H(\Bbbk)$.

For each positive integer $d$ and each integer $n\geq2$, the {\em Chen-Gibeny-Krashen space} $\mathrm{CGK}_{V}(n)$ is a smooth, projective $\Bbbk$-variety parametrizing isomorphism classes of stable $n$-trees of projective $n$-spaces. When $V=\Bbbk^d$, we denote the associated Chen-Gibney-Krashen space $\mathrm{CGK}_{d}(n)$. When $d=1$, i.e., when $V=\Bbbk$, the projective space $H=\mathbb{P}(V)$ is simply a $\Bbbk$-point, and we have an identity
$$
\mathrm{CGK}_{1}(n)=\overline{M}_{0,n+1}.
$$
It will be helpful, for our purposes, to define one additional Chen-Gibney-Krashen moduli space, namely $\mathrm{CGK}_{d}(1)=\Spec{\Bbbk}$.

\subsection{Chen-Gibeny-Krashen operads}
In \cite{West}, C. Westerland points out that for each positive integer $d$, the family of Chen-Gibney-Krashen spaces $\big\{\mathrm{CGK}_{d}(n)\big\}_{n=1}^{\infty}$ forms an operad in the category of $\Bbbk$-schemes. Its operadic structure has a natural description in terms of the stably marked trees of projective $d$-spaces. Namely, for $1\geq i\geq m$, the composition morphism
$$
\gamma_{i}:\mathrm{CGK}_{V}(m)\times\mathrm{CGK}_{V}(n)\lra\mathrm{CGK}_{V}(m+n-1)
$$
takes any stable $m$-marked tree $(B,\jmath,q)$ of projective $d$-spaces, and any stable $n$-marked tree $(B',\jmath',q')$ of projective $d$-spaces, proceeds to blow-up $B$ at its marked point $q_{i}$, and then glue $B'$ to this blow-up $\mathscr{B}\ell_{q_i}B$ by identifying the divisor $H'$ in $B'$ with the exceptional divisor in $\mathscr{B}\ell_{q_i}B$. The result is a stable $(m+n-1)$-tree of projective $d$-space $(B'',\jmath'',q'')$.

\subsection{Fulton-MacPherson compactifications \& families of screens}
Fix a smooth, $d$-dimensional $\Bbbk$-variety $X$, a $\Bbbk$-point $x$ in $X$, and an identification $V\cong \mathrm{T}_{x}^{\ast}X$ with the cotangent space at $x$.

For each $n$, the {\em Fulton-MacPherson compactification} $X[n]$ is a compactification of the moduli space of $n$-tuples of distinct marked points in $X$. Fulton and MacPherson introduce these spaces in \cite{FM}, in order to obtain further algebraic invariants of $X$. The variety $X[n]$ comes with a  proper morphism
$$
\pi:X[n]\lra\!\!\!\!\rightarrow X^{n}
$$
and the Chen-Gibney-Krashen-space $\mathrm{CGK}_{V}(n)$ can be realized as the fiber of $\pi$ over the $\Bbbk$-point $(x,\dots,x)$ in the diagonal $\Delta:X\hookrightarrow X^n$.

The Fulton-MacPherson compactification $X[n]$ is itself a moduli space in the category of $X^n$-schemes. It classifies objects over $X^n$ called "compatible families of screens." For each subset $J\sub \{1,\dots,n\}$ containing at least two elements, let $\mathscr{I}_{J}$ denote the ideal sheaf of the diagonal $\Delta:X\hookrightarrow X^J$, and let $\mathrm{pr}_{J}:X^{n}\lra\!\!\!\!\rightarrow X^{J}$ denote the projection induced by the inclusion $J\hookrightarrow\{1,\dots,n\}$. Then a {\em screen over} $X^J$ is any epimorphism $\mathrm{pr}_{\!J}^{\ast}\mathscr{I}_{J}\lra\!\!\!\!\rightarrow\mathscr{E}_{J}$ of $\mathscr{O}_{X^n}$-modules whose codomain $\mathscr{E}_{J}$ is invertible, i.e., is locally free of rank 1. A {\em compatible family of screens over} $X^n$ is any collection of screens $\mathrm{pr}_{\!J}^{\ast}\mathscr{I}_{J}\lra\!\!\!\!\rightarrow\mathscr{E}_{J}$, one for each subset $J\sub I$ with $|J|\geq2$, such that for each proper inclusion $J\subset J'$ of such subsets in $I$, there exists a morphism making the diagram
$$
\xy
(0,0)*+{\mathrm{pr}_{\!J}^{\ast}\mathscr{I}_{J}}="0";
(0,15)*+{\mathrm{pr}_{\!J'}^{\ast}\mathscr{I}_{J'}}="1";
(20,0)*+{\mathscr{E}_{J}}="2";
(20,15)*+{\mathscr{E}_{J'}}="3";
{\ar@{^{(}->} "0"; "1"};
{\ar@{->>} "1"; "3"};
{\ar@{->>} "0"; "2"};
{\ar@{-->} "2"; "3"};
\endxy
$$
commute. This morphism $\mathscr{E}_{J}\dashrightarrow\mathscr{E}_{J'}$ can be zero, but it is unique if it exists.

The Fulton-MacPherson compactification $X[n]$ is the moduli space of isomorphism classes of compatible families of screens over $X^n$. Thus the realization of the Chen-Gibney-Krashen space as the fiber over $(x,\dots,x)\in X^I$ gives us a second description of $\mathrm{CGK}_{d}(n)$ as a moduli space parametrizing compatible families of screens in a tangent neighborhood of the point $(x,\dots,x)\in X^I$  (for details, see \cite[section 3]{CGK}).

\subsection{A question}
It is natural to ask what form the operadic structure on the family $\big\{\mathrm{CGK}_{d}(n)\big\}_{n=1}^{\infty}$ takes in terms of this latter description each of the varieties $\mathrm{CGK}_{V}(n)$ as a moduli space of compatible families of screens.

Our answer to this question, constituting the bulk of this text, is that the passage from the datum consisting of $X$ and a $\Bbbk$-point $x$ in $X$, to the operad $\big\{\mathrm{CGK}_{d}(n)\big\}_{n=1}^{\infty}$, is just the rank-1 part of one example of a far more general, categorical construction that builds operads from functors.

We call this construction the {\em Fulton-MacPherson construction}. Describing it in the appropriate generality requires that we introduce an abstraction of the concept of a set-valued operad. If $\Cc$ is any category admitting a calculus whereby we can "substitute values into arguments" in $\Cc$, then there exists a corresponding, abstract version of the notion of an operad on $\Cc$. The Fulton-MacPherson construct takes any functor $F:\Cc\lra\Dc$ and, under suitable hypothesis on $\Cc$ and $F$, uses $F$ to build an abstract operad on $\Cc$.

\subsection{Outline of the text}
In Section \ref{abstract operads}, we introduce abstract operads in detail, and give several simple examples.

In Section \ref{FM construction}, we describe the Fulton-MacPherson construction, and we give some basic examples of abstract operads obtained from this construction. Our main result of Sections \ref{abstract operads} and \ref{FM construction} is Theorem \ref{FMop}, which provides conditions under which the Fulton-MacPherson construction produces an abstract, set-valued operad on $\Cc$.

In Section \ref{q grass}, we move on to more heavy-duty examples of abstract operads obtained from Fulton-MacPherson constructions. Our primary focus in Section \ref{q grass} is the study of algebraic varieties that represent these abstract operads.
Specifically, we show that many examples of abstract operads obtained from Fulton-MacPherson constructions are represented by a particular type of projective $\Bbbk$-variety called a {\em quiver Grassmannian}. The simplest example of a quiver Grassmannian is any classical Grassmannian. More generally, a quiver Grassmannian is the reduced moduli space parametrizing subrepresentations of a given representation of a given quiver. In Section \ref{q grass} we include a review of quivers, their representations, and of quiver Grassmannians for readers less familiar with this theory.

Each of the Chen-Gibney-Krashen operads $\mathrm{CGK}_{d}$ is merely the "rank-1" connected component of a much larger operad $\mathrm{FM}_{V}$ in projective $\Bbbk$-varieties. This larger operad $\mathrm{FM}_{V}$ represents a set-valued operad obtained via a particular Fulton-MacPherson construction. We realize each variety $\mathrm{FM}_{V}(n)$ as a finite disjoint union of quiver Grassmannians.
Our main results in Section \ref{q grass} are Theorem \ref{representability}, which formally states this representability. We also give other examples of quiver Grassmannians representing abstract operads obtained from Fulton-MacPherson constructions, and we indicate how to alter the proof of Theorem \ref{representability} in these cases.

\vskip 2cm

\section{Abstract operads}\label{abstract operads}

\vskip .5cm

\subsection{Notation.}
Given a category $\mathcal{X}$, we let
$\mathcal{X}_{\non{iso}}\sub\mathcal{X}$
denote the category with the same objects as $\mathcal{X}$, but with morphisms consisting only of the isomorphisms in $\mathcal{X}$.

\subsection{Abstract operads.}

Fix a symmetric monoidal category
$(\mathcal{M},\tensor,\mathds{1},\alpha,\tau)$,
with associator
 $\alpha:(M_1\tensor M_2)\tensor M_3\stackrel\sim\longrightarrow M_1\tensor(M_2\tensor M_3)$ 
and symmetry transformation
$\tau: M_1\tensor M_2\stackrel\sim\longrightarrow M_2\tensor M_1$.

Let $\mathcal{C}$ be a category containing a terminal object $\ast$.
We refer to $\mathcal{C}$ as our {\em input category}, and we refer to any commutative square in $\Cc$ of the form
$$\xymatrix{
C_2\ar[r]\ar[d]
&C_3\ar[d]\cr
\ast\ar@{^{(}->}_{c_{1}}[r]
&C_1
}$$
as an {\em acute} square.

An acute square in
$\mathcal{C}$
that happens to be bicartesian is the categorical version of a substitution of the object $C_2$ for the point $c_1$ in the object $C_1$.
For arbitrary sets $I$ and $J$,
the substitution of $J$ for an element $i$ in $I$ is the fundamental operation underlying the structure of any operad. This leads us to the following abstraction:

\begin{definition}\label{absop}
{\bf (Abstract operad)}
An {\em abstract} $\mathcal{M}$-{\em valued operad on} $\mathcal{C}$ is a datum consisting of:
\begin{itemize}
\item[{\bf (i)}]
A functor $\mathcal{P}:\isop{\Cc}\longrightarrow \mathcal{M}$;
\item[{\bf (ii)}]
An assignment $\gamma$ that associates, to each acute bicartesian square
$$\xymatrix{
\ar@{}[dr]|{\non{(a)}}
C_2\ar[r]\ar[d]
&C_3\ar[d]\cr
\ast\ar@{^{(}->}[r]
&C_1
}$$
in $\mathcal{C}$, a morphism
$\gamma_{\non{(a)}}:\mathcal{P}(C_1)\tensor\mathcal{P}(C_2)\longrightarrow\mathcal{P}(C_3)$
in $\mathcal{M}$.
\end{itemize}
This datum must satisfy three axioms:

\vskip .2cm

{\bf OP-1 (Equivariance)} For each isomorphism between bicartesian squares (a) and (b) in
$\mathcal{C}$, given by a commutative diagram
$$
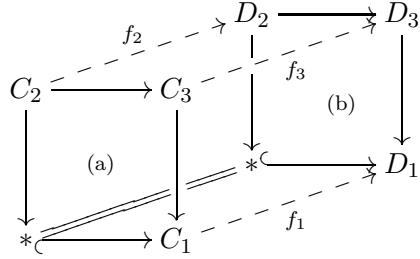
\xy
(0,0)*+{\ast}="*c";
(30,10)*+{\ast}="*d";
(0,20)*+{C_2}="c2";
(20,20)*+{C_3}="c3";
(20,0)*+{C_1}="c1";
(30,30)*+{D_2}="d2";
(50,30)*+{D_3}="d3";
(50,10)*+{D_1}="d1";
(10,10)*+{{}_{\non{(a)}}}="(a)";
(42,18)*+{{}_{\non{(b)}}}="(b)";
{\ar@{_{(}->} "*c";"c1"};
{\ar@{->} "c2";"*c"};
{\ar@{->} "c2";"c3"};
{\ar@{->} "c3";"c1"};
{\ar@{^{(}->} "*d";"d1"};
{\ar@{->}|(.34){\hole} "d2";"*d"};
{\ar@{->} "d2";"d3"};
{\ar@{->} "d3";"d1"};
{\ar@{=}|(.675){\hole} "*c";"*d"};
{\ar@{-->}^{f_2} "c2";"d2"};
{\ar@{-->}_{f_3} "c3";"d3"};
{\ar@{-->}_{f_1} "c1";"d1"};
\endxy
$$
in $\Cc$, with $f_1$, $f_2$, and $f_3$ isomorphisms,
the corresponding diagram in
$\mathcal{M}$, namely
$$\xymatrix{
\mathcal{P}(C_1)\tensor\mathcal{P}(C_2)\ar[r]^-{\gamma_{\non{(a)}}}
&\mathcal{P}(C_3)\cr
\mathcal{P}(D_1)\tensor\mathcal{P}(D_2)\ar[r]_-{\gamma_{\non{(b)}}}
\ar[u]^-{\mathcal{P}(f_1)\tensor\mathcal{P}(f_2)}
&\mathcal{P}(D_3)
\ar[u]_-{\mathcal{P}(f_3)},
}
\ \ \ \ \ \ \ \ \ \ \ \ 
$$
is commutative.

\vskip .2cm

{\bf OP-2 (Associativity)}
For each commutative diagram
$$
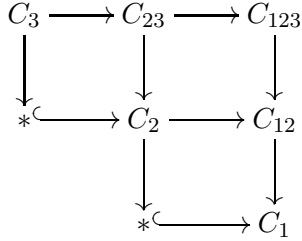
\xymatrix{
C_3\ar[r]\ar[d]
&C_{23}\ar[r]\ar[d]
&C_{123}\ar[d]\\
\ast\ar@{^{(}->}[r]
&C_2\ar[r]\ar[d]
&C_{12}\ar[d]\\
\ &\ast\ar@{^{(}->}[r]
&C_1
}
\ \ \ \ \ \ 
$$
in
$\mathcal{C}$,
consisting entirely of bicartesian squares,
the compositions
$\gamma_{\non{(a)}}$,
$\gamma_{\non{(b)}}$,
$\gamma_{\non{(c)}}$,
and
$\gamma_{\non{(d)}}$
associated to the acute bicartesian squares (a), (b), (c), and (d) in
$$
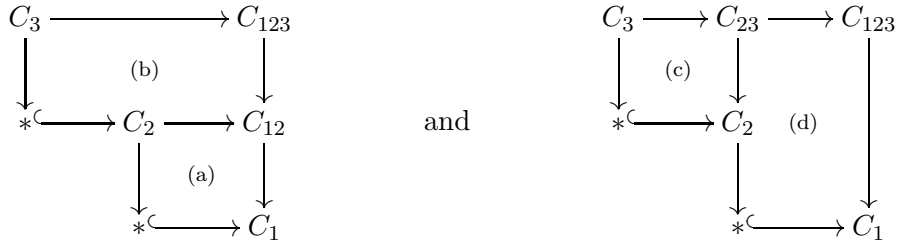
\begin{array}[c]{lcr}
\vcenter{\vbox{
\xymatrix{
C_3\ar[rr]\ar[d]\ar@{}[rrd]|{\non{(b)}}
&&C_{123}\ar[d]\\
\ast\ar@{^{(}->}[r]
&\ar@{}[dr]|{\non{(a)}}
C_2\ar[r]\ar[d]
&C_{12}\ar[d]\\
\ &\ast\ar@{^{(}->}[r]
&C_1}
}}
&
\mathrm{\ \ \ \ \ \ \ \ \ \ and\ \ \ \ \ \ \ \ \ \ }
&
\vcenter{\vbox{
\xymatrix{
\ar@{}[rd]|{\non{(c)}}
C_3\ar[r]\ar[d]
&C_{23}\ar[r]\ar[d]
\ar@{}[rdd]|{\non{(d)}}
&C_{123}\ar[dd]\\
\ast\ar@{^{(}->}[r]
&C_2\ar[d]&&
\\
\ &\ast\ar@{^{(}->}[r]
&C_1
}
}}
\end{array}
$$
fit into the following commutative diagram in $\mathcal{M}$:
$$\xymatrix{
\Big(\mathcal{P}(C_1)\tensor\mathcal{P}(C_2)\Big)\tensor\mathcal{P}(C_3)
\ar[rr]^{\ \ \ \ \ \ \gamma_{\non{(a)}}\tensor\non{id}}
\ar[dd]_{\alpha}
&&
\mathcal{P}(C_{12})\tensor\mathcal{P}(C_3)
\ar[rd]^{\gamma_{\non{(b)}}}\\
&&&\mathcal{P}(C_{123})
\\
\mathcal{P}(C_1)\tensor\Big(\mathcal{P}(C_2)\tensor\mathcal{P}(C_3)\Big)
\ar[rr]_{\ \ \ \ \ \ \non{id}\tensor\gamma_{\non{(c)}}}
&&
\mathcal{P}(C_1)\tensor\mathcal{P}(C_{23})
\ar[ru]_{\gamma_{\non{(b)}}}
}$$

\vskip .2cm

{\bf OP-3 (Locality)}
Consider any eight-faced polyhedral diagram in
$\mathcal{C}$,
glued from commutative diagrams
$$
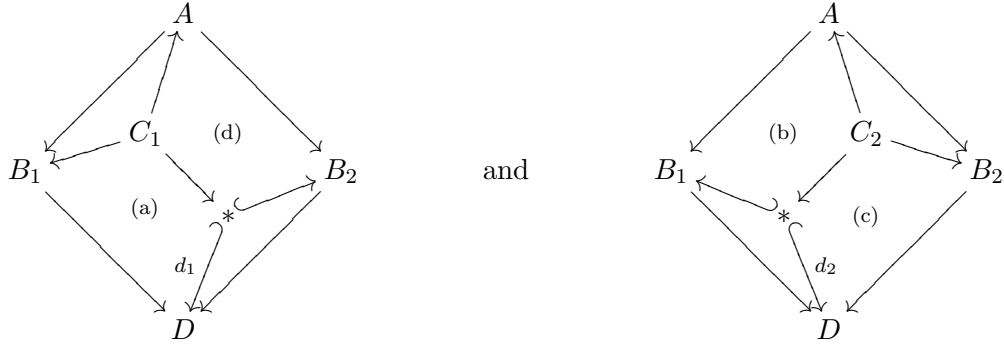
\vcenter{\vbox{
\xy
(21,42)*+{A}="a";
(0,21)*+{B_{1}}="b1";
(42,21)*+{B_{2}}="b2";
(21,0)*+{D}="d";
(16,26)*+{C_{1}}="c1";
(27,15)*+{\ast}="*";
(16,16)*+{{}_{\non{(a)}}}="(a)";
(27,26)*+{{}_{\non{(d)}}}="(d)";
{\ar@{->} "a";"b1"};
{\ar@{->} "a";"b2"};
{\ar@{->} "b1";"d"};
{\ar@{->} "b2";"d"};
{\ar@{->} "c1";"a"};
{\ar@{->} "c1";"b1"};
{\ar@{->} "c1";"*"};
{\ar@{^{(}->} "*";"b2"};
{\ar@{_{(}->}_{d_1} "*";"d"}
\endxy
}}
\ \ \ \ \ \ \ \ \ \ \ \ 
\mathrm{and}
\ \ \ \ \ \ \ \ \ \ \ \ 
\vcenter{\vbox{
\xy
(21,42)*+{A}="a";
(0,21)*+{B_{1}}="b1";
(42,21)*+{B_{2}}="b2";
(21,0)*+{D}="d";
(26,26)*+{C_{2}}="c2";
(15,15)*+{\ast}="*";
(26,15)*+{{}_{\non{(c)}}}="(c)";
(15,26)*+{{}_{\non{(b)}}}="(b)";
{\ar@{->} "a";"b1"};
{\ar@{->} "a";"b2"};
{\ar@{->} "b1";"d"};
{\ar@{->} "b2";"d"};
{\ar@{->} "c2";"a"};
{\ar@{->} "c2";"b2"};
{\ar@{->} "c2";"*"};
{\ar@{_{(}->} "*";"b1"};
{\ar@{^{(}->}^{d_2} "*";"d"}
\endxy
}}$$
by identifying their boundary squares.
Suppose that
$d_1\not= d_2$,
and that the squares (a), (b), (c), and (d) are bicartesian.
Then the corresponding diagram in
$\mathcal{M}$, namely
$$\xymatrix{
\Big(\mathcal{P}(D)\tensor\mathcal{P}(C_1)\Big)\tensor\mathcal{P}(C_2)
\ar[rr]^{\ \ \ \ \ \ \gamma_{\non{(a)}}\tensor\non{id}}
\ar[dd]_{\alpha^{-1}\circ\tau\circ\alpha}
&&
\mathcal{P}(B_1)\tensor\mathcal{P}(C_2)
\ar[rd]^{\gamma_{\non{(b)}}}\\
&&&\mathcal{P}(A)
\\
\Big(\mathcal{P}(D)\tensor\mathcal{P}(C_2)\Big)\tensor\mathcal{P}(C_1)
\ar[rr]_{\ \ \ \ \ \ \ \gamma_{\non{(c)}}\tensor\non{id}}
&&
\mathcal{P}(B_2)\tensor\mathcal{P}(C_1)
\ar[ru]_{\gamma_{\non{(b)}}}
}$$
must be commutative.
\end{definition}

\vskip .2cm

\begin{definition}
{\bf (Morphisms and categories of abstract operads.)}
If
$\mathcal{P}$
and
$\mathcal{Q}$
are abstract
$\mathcal{M}$-valued operads on
$\mathcal{C}$,
with respective composition assignments
$\gamma$ and $\delta$,
then a {\em morphism of operads}
$\varphi:\mathcal{P}\longrightarrow\mathcal{Q}$
is any natural transformation
$\varphi:\mathcal{P}\Longrightarrow\mathcal{Q}$
that preserves composition,
in the sense that for each acute bicartesian square (a) in
$\mathcal{C}$,
the corresponding diagram
$$\xymatrix{
\mathcal{P}(C_1)\tensor\mathcal{P}(C_2)
\ar[d]_-{\varphi_{C_1}\tensor\varphi_{C_2}}
\ar[r]^-{\gamma_{\non{(a)}}}
&\mathcal{P}(C_3)\ar[d]^-{\varphi_{C_3}}
\cr
\mathcal{Q}(C_1)\tensor\mathcal{Q}(C_2)\ar[r]_-{\delta_{\non{(a)}}}
&\mathcal{Q}(C_3),
}$$
in
$\mathcal{M}$
commutes.

We let $\bold{Op}(\Cc,\Mc)$
denote the category of $\Mc$-valued operads on $\Cc$.

\end{definition} 

\vskip .2cm

\begin{definition}
{\bf (Right module over an abstract operad.)}
If $\mathcal{P}$
is an abstract
$\mathcal{M}$-valued operad on
$\mathcal{C}$,
then a {\em right} $\mathcal{P}$-{\em module}
is a functor
$\mathcal{R}:\isop{C}\longrightarrow \mathcal{M}$
that associates a morphism
$\varrho_{\non{(a)}}:\mathcal{R}(C_1)\tensor\mathcal{P}(C_2)\longrightarrow\mathcal{R}(C_3)$
in $\mathcal{M}$
to each acute bicartesian square (a) in $\mathcal{C}$,
satisfying the obvious analogs of axioms OP-1 through OP-3 above.

A right $\mathcal{P}$-module $\mathcal{R}$ is {\em constant}
if, for some fixed object $R$ in $\mathcal{M}$,
we have $\mathcal{R}(C)=R$ for every object $C$ in $\mathcal{C}$,
such that $\mathcal{R}$ takes every morphism in $\isop{\mathcal{C}}$ to the identity on $R$.

We define morphisms $\varphi:\mathcal{R}_1\longrightarrow\mathcal{R}_2$ between right $\mathcal{P}$-modules in the obvious way.
We let $\bold{Mod}_\mathcal{P}$ denote the category of right $\mathcal{P}$-modules, and we let $\mathcal{M}_\mathcal{P}$ denote the category of constant right $\mathcal{P}$-modules.

\end{definition}

\vskip .2cm

\subsection{Examples} We now list some immediate examples of abstract operads.

\vskip .2cm

\begin{example}

{\bf (Classical operads.)}

In the category $\mathcal{C}=\fnS$ of nonempty finite sets, every acute bicartesian diagram is isomorphic to one of the form
$$\xymatrix{
J\ar@{^{(}->}[r]\ar@{->>}[d]
&J\sqcup(I-\{ i \})\ar@{->>}[d]^{p}\cr
\bold{*}\ar@{^{(}->}[r]_i
&I
}$$
where the function $p$ identifies $I-\{ i \}$ with the corresponding subset of $I$, and maps the whole subset $J$ to the element $i\in I$.
Thus composition in any abstract operad
$\mathcal{P}:\isop{\fnS}\rightarrow\mathcal{M}$
takes its complete determination from a family morphisms
$$\gamma_i:\mathcal{P}(I)\tensor\mathcal{P}(J)\rightarrow\mathcal{P}(J\sqcup(I-\{ i \}))$$
and our Definition \ref{absop} recovers the definition of a classical,
$\mathcal{M}$-valued operad, as axiomatized by Markl in \cite{M}.

In terms of the analogy that exists between rooted trees and operads, axiom {\bf OP-2} in this case corresponds to the fact that the order in which we graft any three trees
$$
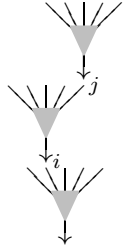
\xy
(-1,11)*+{{}_{i}};
(4,21)*+{{}_{j}};
(0,5)*=0{}; (-5,10)*=0{} **@{-},
(0,5)*=0{}; (-2.5,10)*=0{} **@{-},
(0,5)*=0{}; (0,10)*=0{} **@{-},
(0,5)*=0{}; (2.5,10)*=0{} **@{-},
(0,5)*=0{}; (5,10)*=0{} **@{-},
(-2.5,16)*=0{}; (-7.5,21)*=0{} **@{-},
(-2.5,16)*=0{}; (-5,21)*=0{} **@{-},
(-2.5,16)*=0{}; (-2.5,21)*=0{} **@{-},
(-2.5,16)*=0{}; (0,21)*=0{} **@{-},
(-2.5,16)*=0{}; (2.5,21)*=0{} **@{-},
(2.5,27)*=0{}; (-2.5,32)*=0{} **@{-},
(2.5,27)*=0{}; (0,32)*=0{} **@{-},
(2.5,27)*=0{}; (2.5,32)*=0{} **@{-},
(2.5,27)*=0{}; (5,32)*=0{} **@{-},
(2.5,27)*=0{}; (7.5,32)*=0{} **@{-},
(0,1.5)*=0{\ \!\!\ \!\!\ \!\!\ \!\!\downarrow}; (0,5)*=0{} **@{-},
(2.5,23.5)*=0{\ \!\!\ \!\!\ \!\!\ \!\!\downarrow}; (2.5,27)*=0{} **@{-},
(-2.5,12.5)*=0{\ \!\!\ \!\!\ \!\!\ \!\!\downarrow}; (-2.5,16)*=0{} **@{-},
(0,5)*+{\btd};
(2.5,27)*+{\btd};
(-2.5,16)*+{\btd};
\endxy
$$
does not effect the resulting tree, and axiom {\bf OP-3} corresponds to fact that the order in which we graft any three trees
$$
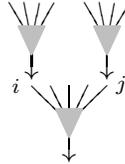
\xy
(-7,10)*+{{}_{i}};
(7,10)*+{{}_{j}};
(0,1.5)*=0{\ \!\!\ \!\!\ \!\!\ \!\!\downarrow}; (0,5)*=0{} **@{-},
(0,5)*=0{}; (-5,10)*=0{} **@{-},
(0,5)*=0{}; (-2.5,10)*=0{} **@{-},
(0,5)*=0{}; (0,10)*=0{} **@{-},
(0,5)*=0{}; (2.5,10)*=0{} **@{-},
(0,5)*=0{}; (5,10)*=0{} **@{-},
(-5,12.5)*=0{\ \!\!\ \!\!\ \!\!\ \!\!\downarrow}; (-5,16)*=0{} **@{-},
(-5,16)*=0{}; (-6,21)*=0{} **@{-},
(-5,16)*=0{}; (-8,21)*=0{} **@{-},
(-5,16)*=0{}; (-4,21)*=0{} **@{-},
(-5,16)*=0{}; (-2,21)*=0{} **@{-},
(5,12.5)*=0{\ \!\!\ \!\!\ \!\!\ \!\!\downarrow}; (5,16)*=0{} **@{-},
(5,16)*=0{}; (6,21)*=0{} **@{-},
(5,16)*=0{}; (8,21)*=0{} **@{-},
(5,16)*=0{}; (4,21)*=0{} **@{-},
(5,16)*=0{}; (2,21)*=0{} **@{-},
(0,5)*=0{\btd};
(-5,16)*=0{\btd};
(5,16)*=0{\btd}
\endxy
$$
does not effect the resulting tree.

We also have natural variants where $\mathcal{C}=\non{F\Sets}$,
the category of finite sets.
For example, if we fix a commutative ring $R$ and let
$\mathcal{M}=R$-$\bold{Mod}$, the category of $R$-modules equipped with its Cartesian monoidal structure,
then the operad
$$\mathcal{C}om_R:\fnS_{\non{iso}}^{\non{op}}\longrightarrow R\textendash\bold{Mod}$$
of commutative rings has a natural extension
$$\mathcal{C}om_R:{\non{F\Sets}}_{\non{iso}}^{\non{op}}\longrightarrow R\textendash\bold{Mod}$$
given by setting $\mathcal{C}om_R({\O})=R$,
and associating, to each acute bicartesian square
$$\xymatrix{
{\O}\ar@{^{(}->}[r]\ar@{^{(}->}[d]
&{I-\{ i \}}\ar@{^{(}->}[d]\cr
\bold{*}\ar@{^{(}->}[r]_i
&I
}$$
in $\non{F\Sets}$, the composition
$$R\tensor_{R}\mathcal{C}om_R(I)\longrightarrow\mathcal{C}om_R({I-\{ i \}})$$
that substitutes each scalar $r\in R$ into the $i^{\non{th}}$-place of the $I$-ary operation that generates
$\mathcal{C}om_R(I)$.

\end{example}

\vskip .2cm

\begin{example}

{\bf (Terminal operads and Grothendieck semigroups.)}
Every category $\mathcal{C}$ with terminal object $\ast$ admits at least one operad, namely the {\em terminal operad}
$\mathcal{K}_\mathcal{C}:\mathcal{C}_{\non{iso}}^{\non{op}}\longrightarrow \Sets$,
which takes each object $C$ in $\mathcal{C}$ to the set
$\{ [C] \}$
containing only one element $[C]$.

Let $\mathcal{C}$ be an essentially small abelian category
$\mathcal{C}=\mathcal{A}$,
so that the collection
$$
\non{ob}(\mathcal{A})/\non{iso}
$$
of isomorphism classes of objects in
$\mathcal{A}$
is a set.
Then we can form the {\em Grothendieck semigroup} ${K}_0^{+}(\mathcal{A})$,
defined as the quotient of the free semigroup
$\non{Fr}(\non{ob}(\mathcal{A})/\non{iso})$
by the relation "$\sim$" whereby
$[A_1]+[A_2]\sim [A_3]$
whenever there exists a short exact sequence
$$0\longrightarrow A_2\longrightarrow A_3\longrightarrow A_1\rightarrow 0$$
in $\mathcal{A}$.
Since an acute bicartesian square in $\mathcal{A}$
is the same thing as short exact sequence in $\mathcal{A}$,
the associativity axiom {\bf OP-2} implies that the category
$\Sets_{\mathcal{K}_\mathcal{A}}$
of constant right-modules for the terminal operad
$\mathcal{K}_\mathcal{A}$
is equivalent to the category of sets equipped with a right
${K}_0^{+}(\mathcal{A})$-action.

In this way, the terminal operad
$\mathcal{K}_\mathcal{C}$
on an arbitrary category
$\mathcal{C}$
(with terminal object $\ast$)
can be seen as a kind of "nonabelian Grothendieck semigroup" of
$\mathcal{C}$.
Since every set-valued operad
$\mathcal{P}:\isop{\Cc}\longrightarrow \Sets$
admits a unique morphism
$\mathcal{P}\longrightarrow\mathcal{K}_\mathcal{C}$,
this turns the theory of abstract set-valued operads on
$\mathcal{C}$
into a complicated type of nonabelian
$K_{0}^{+}$-theory for
$\mathcal{C}$.

\end{example}

\vskip .2cm

\begin{example}

{\bf (Charades.)}
If we let our input category $\Cc$ be a pointed category,
that is, a category in which the terminal object $\ast$ is also an initial object,
then axiom {\bf OP-3} for any abstract operad
$\mathcal{P}:\isop{\Cc}\longrightarrow\mathcal{M}$
becomes vacuous, since the hypothesis $d_1\not=d_2$ of {\bf OP-3} can never be satisfied in a pointed category.

This happens, for instance, when $\Cc=\Ac$ is any abelian category.
When $\mathcal{M}=\bold{Vect}_{\Bbbk}$, the category of finite-dimensional vector spaces over some fixed ground field $\Bbbk$, equipped with its Cartesian monoidal structure, the remaining axioms {\bf OP-1} and {\bf OP-2} recover the axioms of a $\Bbbk$-{\em linear charade} on $\mathcal{A}^\mathrm{op}$, as defined by M. Kapranov in \cite{K}.
Note that our abstract operads on $\Ac$ become charades on $\Ac^\mathrm{op}$, and vice versa.
More generally, we can consider charades taking values in any symmetric monoidal category $\Mc$. Let us briefly mention two examples:

Fix a finite field $\Bbbk=\mathbb{F}_q$. Let our input category be $\mathcal{C}=\non{F\bold{Vect}}_\Bbbk^{\non{op}}$,
and let $\mathcal{M}=\non{F\bold{Vect}}_\Bbbk$.
For each finite dimensional $\Bbbk$-vector space $V$,
let $\non{Bld}_{\bullet}(V)$
denote the Tits building associated to $V$.
This is a simplicial set with set of $n$-simplices
$$\non{Bld}_{n}(V)=\{\non{towers}\ W_0\sub\cdots\sub W_n\ \non{of\ proper\ non\textendash zero\ subspaces\ of}\ V \}.$$
Then $V$'s {\em Steinberg module} is the $\Bbbk$-vector space
$\mathcal{S}t(V)=H_{\non{dim}(V)-1}(\non{Bld}_{\bullet}(V),\Bbbk)$.
It constitutes a functor
$\mathcal{S}t:(\non{F\bold{Vect}}_\Bbbk^{\non{op}})_{\non{iso}}^{\non{op}}=\non{F\bold{Vect}}_\Bbbk\longrightarrow\non{F\bold{Vect}}_\Bbbk$,
and comes with a naturally defined composition that turns it into a charade. See \cite[section 3.3]{K} for details.

For another example, 
fix a scheme $X$,
let $\mathcal{C}=\bold{Bun}_X$
be the category of locally free, finite rank
$\mathcal{O}_X$-modules,
and let $\mathcal{M}=\bold{Pic}_X$
be the Picard category of $X$.
Let
$\mathcal{D}et:(\bold{Bun}_X)_{\non{iso}}\longrightarrow\bold{Pic}_X$
be the functor that, at each locally free $\mathcal{O}_X$-module $\mathcal{B}$, returns the line bundle
$\mathcal{D}et(\mathcal{B})$
given by the top exterior power of $\mathcal{B}$.
Then $\mathcal{D}et$ is an example of a charade.

\end{example}

\vskip .7cm

\vfill\eject

\section{The Fulton-MacPherson construction.}\label{FM construction}

\vskip .5cm

\subsection{Notation and terminology.}
If $\mathcal{X}$ is an arbitrary category and $X$ is any object in $\mathcal{X}$,
let $\bold{Mon}_X\mathcal{X}$
denote the category whose objects are monomorphisms
$Y\hookrightarrow X$
in $\mathcal{X}$,
and whose arrows are morphisms
$Y_1\hookrightarrow Y_2$
over $X$ (necessarily monomorphisms).
We refer to $\bold{Mon}_X\mathcal{X}$
as the {\em poset of subobjects of} $X$.

Similarly, let
$\bold{Epi}^{X}\mathcal{X}$
denote the category whose objects are epimorphisms
$X\lra\!\!\!\!\rightarrow Y$
in $\mathcal{X}$,
and whose arrows are morphisms
$Y_1\lra\!\!\!\!\rightarrow Y_2$
under $X$ (necessarily epimorphisms).

\subsection{The Fulton-MacPherson Construction.} 
Suppose given a functor $F:\mathcal{C}\longrightarrow\mathcal{D}$,
running from our input category $\mathcal{C}$
to an arbitrary category $\mathcal{D}$.
Then the {\em Fulton}-{\em MacPherson construction}
for $F$ proceeds as follows:

At each object $C$ in $\mathcal{C}$, form the poset 
$\bold{Mon}_C\mathcal{C}$
of subobjects of $C$.
This category comes with a forgetful functor
$U_C:\bold{Mon}_C\mathcal{C}\longrightarrow\mathcal{C}$
that takes each object
$B\hookrightarrow C$
in $\bold{Mon}_C\mathcal{C}$
to the object $B$ in $\mathcal{C}$.
Let $F_C:\bold{Mon}_C\mathcal{C}\longrightarrow\mathcal{D}$
denote the composite
$$F_C:\ \ \bold{Mon}_C\mathcal{C}\xrightarrow{\ \ U_{C}\ }\mathcal{C}\xrightarrow{\ \ F\ \ }\mathcal{D}$$
so that $F_{C}(B\hookrightarrow C)=F(B)$.
This composite $F_C$ is an object in the functor category
$[\bold{Mon}_C\mathcal{C},\mathcal{D}]$,
and we can therefore consider the totality of epimorphisms
$F_C\Longrightarrow\!\!\!\!\!\!\!\!\Longrightarrow E$ in $[\bold{Mon}_C\mathcal{C},\mathcal{D}]$.
These epimorphisms form a category, namely
$$\bold{Epi}^{F_C}[\bold{Mon}_C\mathcal{C},\ \mathcal{D}]$$
Without additional assumptions on $\Cc$ or $F$, we have no guarantee that the isomorphism classes of objects in this last category
$\bold{Epi}^{F_C}[\bold{Mon}_C\mathcal{C},\ \mathcal{D}]$
will form an actual set. If and when they do form an honest set, we denote this set
$\mathcal{F\! M}_F(C)$. Thus
$$\mathcal{F\! M}_F(C)\ \ =\ \ \non{ob}\left(\bold{Epi}^{F_C}[\bold{Mon}_C\mathcal{C},\ \mathcal{D}]\right)/\non{isomorphism}$$

Now suppose that $C$ and $C'$
are objects in $\mathcal{C}$
that both return honest sets
$\mathcal{F\! M}_F(C)$
and
$\mathcal{F\! M}_F(C')$,
and suppose that
$a:C\xrightarrow{\ \sim\ }C'$
is an isomorphism in $\mathcal{C}$.
Then $a$
induces an equivalence of categories
$a_*:\bold{Mon}_C\mathcal{C}\xrightarrow{\ \sim\ }\bold{Mon}_{C'}\mathcal{C}$,
which takes each object
$B\hookrightarrow C$
in $\bold{Mon}_C\mathcal{C}$
to the object
$B\hookrightarrow C\xrightarrow{\sim}C'$
in
$\bold{Mon}_{C'}\mathcal{C}$.
Since $F_{C'}{}_{{}^{{\ }^\circ}}a_*=F_C$ (on the nose),
this equivalence $a_\ast$ induces a functor
$$a^*:\ \ \bold{Epi}^{F_{C'}}[\bold{Mon}_{C'}\mathcal{C},\ \mathcal{D}]
\ \longrightarrow\ 
\bold{Epi}^{F_{C}}[\bold{Mon}_{C}\mathcal{C},\ \mathcal{D}]$$
that takes each epimorphism
$F_{C'}\Longrightarrow\!\!\!\!\!\!\!\!\Longrightarrow E$
in $[\bold{Mon}_{C'}\mathcal{C},\ \mathcal{D}]$
to the epimorphism
$F_{C}\Longrightarrow\!\!\!\!\!\!\!\!\Longrightarrow Ea_*$
in $[\bold{Mon}_{C}\mathcal{C},\ \mathcal{D}]$.

This means that if every single object $C$ in $\mathcal{C}$
returns an honest set
$\mathcal{F\! M}_F(C)$,
then we have a functor
$$\mathcal{F\! M}_F:\mathcal{C}_{\non{iso}}^{\non{op}}\longrightarrow\Sets.$$
When this functor is well defined, we call it the {\em Fulton-MacPherson functor}
determined by $F$. When
$\mathcal{C}=\mathcal{D}$
and
$F=\non{Id}_C$,
we denote the Fulton-MacPherson functor simply
$\mathcal{F\! M}:\mathcal{C}_{\non{iso}}^{\non{op}}\longrightarrow\Sets$

In Theorem \ref{FMop} below, we provide conditions on $\Cc$ and $F$ which guarantee not only that the Fulton-MacPherson construction produces an honest functor
$\mathcal{F\! M}_F$,
but that this functor comes with the natural structure of a set-valued operad on
$\mathcal{C}$.

Let us first look at a couple of examples:

\vskip .2cm

 \vskip .2cm

\begin{example}\label{ExA}

{\bf (The terminal operad.)}
Assume that $\mathcal{C}$ has a terminal object $\ast$,
and let $\mathcal{D}=\bold{Pt}$
be the {\em terminal category},
with only one object $\bold{1}$
and only one arrow $\non{id}_\bold{1}:\bold{1}\longrightarrow\bold{1}$.
Then there exists only one functor
$F=\non{pt}:\mathcal{C}\longrightarrow\bold{Pt}$.
Applying the Fulton-MacPherson construction to this functor,
we recover the terminal operad on $\Cc$:
$$\mathcal{F\! M}_{\non{pt}}\cong\mathcal{K}_\mathcal{C}$$

\end{example}

\vskip .2cm

\begin{example}\label{ExB}
{\bf (Decreasing endomorphisms of posets of subobjects.)}
Fix a ground field $\Bbbk$,
let $\mathcal{C}=\mathcal{D}=\non{\textsf{F}\bold{Vect}}_\Bbbk$,
and let
$F=\non{Id}:\non{\textsf{F}\bold{Vect}}_\Bbbk\longrightarrow\non{\textsf{F}\bold{Vect}}_\Bbbk$.
Since the condition of being an epimorphism is determined componentwise for natural transformations between
$\non{\textsf{F}\bold{Vect}}_\Bbbk$-valued
functors,
the Fulton-MacPherson functor
$\mathcal{F\! M}:(\non{\textsf{F}\bold{Vect}}_\Bbbk)_{\non{iso}}^{\non{op}}\longrightarrow\Sets$
is well defined in this case.

Consider a finite dimensional $\Bbbk$-vector space
$V$,
and let $\non{Sub}(V)$ denote the small poset of actual subspaces of $V$.
Kernel/cokernel duality in
$\non{\textsf{F}\bold{Vect}}_\Bbbk$
identitifes each element
$$\pi:\non{Id}_{\bold{Mon}_V(\non{\textsf{F}\bold{Vect}}_\Bbbk)}\Longrightarrow\!\!\!\!\!\!\!\!\Longrightarrow E$$
of
$\mathcal{F\! M}(V)$
with the decreasing endomorphism
$\varepsilon_\pi:\non{Sub}(V)\lra\non{Sub}(V)$
that takes each subspace
$W\sub V$
to the kernel $\varepsilon_{\pi}(W)\sub W$ of the component $\pi_{W}:W
\lra\!\!\!\!\rightarrow E(W)$ of $\pi$.
In this way, we get an isomorphism
$$
\mathcal{F\! M}(V)
\ \cong\ 
\left\{
	\begin{array}{l}
		\mathrm{decreasing\ endomorphisms}
		\\
		\ \mathrm{of\ the\ small\ poset\ Sub}(V)
	\end{array}
\right\}
$$

This Fulton-MacPherson functor
$\mathcal{F\! M}:(\non{\textsf{F}\bold{Vect}}_\Bbbk)_{\non{iso}}^{\non{op}}\longrightarrow\Sets$
admits a composition assignment that turns it into an abstract operad.
Specifically, given any short exact sequence
$$
\non{(a)}:\ \ \ \ 0\longrightarrow V_2\longrightarrow V\xrightarrow{\ \ p\ }V_1\longrightarrow 0\ \ \ \ \ \ \ \ 
$$
in $\FVect{\Bbbk}$,
let
$$
\gamma_{\non{(a)}}:\ \mathcal{F\! M}(V_1)\times\mathcal{F\! M}(V_2)\longrightarrow\mathcal{F\! M}(V)
$$
be the map taking each pair $(\varepsilon_1,\ \varepsilon_2)$,
consisting of decreasing endomorphisms 
$\varepsilon_1$ of $\non{Sub}(V_1)$
and
$\varepsilon_2$ of $\non{Sub}(V_2)$,
to the decreasing endomorphsim
$\gamma_{\non{(a)}}(\varepsilon_1,\ \varepsilon_2)$
of $\non{Sub}(V)$,
which, at each subspace $W\sub V$, returns
$$
\gamma_{\non{(a)}}(\varepsilon_1,\ \varepsilon_2)(W)=
\left\{
	\begin{array}{ll}
		\ \ \ \ \ \ \ \varepsilon_2(W)  & \mbox{if} \ \ W\sub V_2 \\
		W\cap p^{-1}\varepsilon_1(pW) & \mbox{if} \ \ W\nsubseteq V_2
	\end{array}
\right.
$$
One can verify directly that this assignment $\gamma$ satisfies axioms
{\bf OP-1} and {\bf OP-2} (the remaining axiom {\bf OP-3} being vacuous).

\end{example}

\vskip .5cm

\subsection{Operadic structure.}\label{OpStruct}
The Fulton-MacPherson functor
$\mathcal{F\! M}_F:\mathcal{C}_{\non{iso}}^{\non{op}}\longrightarrow\Sets$,
when it is well defined, is "almost" an abstract operad;
it lacks only a composition assignment
$\gamma$
satisfying {\bf OP-1} through {\bf OP-3}.
The Fulton-MacPherson functors we saw in Examples \ref{ExA} and \ref{ExB}
above come with just such an assignment.

The proof of Theorem \ref{FMop} below gives a general construction of at least one rather natural composition assignment on
$\mathcal{F\! M}_F$ that turns $\mathcal{F\! M}_F$ into an abstract operad.
Theorem \ref{FMop} itself provides sufficient conditions on the category
$\mathcal{C}$
and on the functor
$F:\mathcal{C}\longrightarrow\mathcal{D}$
to guarantee that this construction works.

We call the resulting operad the {\em Fulton-MacPherson operad} associated to $F$.

The conditions of Theorem \ref{FMop}, whose terms we will define momentarily, are as follows:

\vskip .2cm

\begin{itemize}
\item[{\bf (\ref{OpStruct}-i)}]
$\mathcal{C}$ must be a {\em good input category};
\item[{\bf (\ref{OpStruct}-ii)}]
The functor $F:\mathcal{C}\longrightarrow\mathcal{D}$
must be {\em right-acute}, with {\em locally defined}, {\em locally small relations}
over $\mathcal{C}$.
\end{itemize}

\vskip .2cm

Before explaining the undefined terms appearing in conditions {\bf (\ref{OpStruct}-i)} and {\bf (\ref{OpStruct}-ii)}, let us formally state our Theorem:

\vskip .2cm

\begin{theorem}\label{FMop}
If conditions {\bf (\ref{OpStruct}-i)} and {\bf (\ref{OpStruct}-ii)} hold for
$\mathcal{C}$ and $F$, then:
\begin{itemize}
\item[{\bf (I)}]
The Fulton-MacPherson functor
$\mathcal{F\! M}_F:\mathcal{C}_{\non{iso}}^{\non{op}}\longrightarrow\Sets$
is well defined;
\item[{\bf (II)}]
There exists a composition assignment
$\gamma$
that turns
$\mathcal{F\! M}_F$
into an abstract operad on
$\mathcal{C}$.
\end{itemize}
\end{theorem}

\vskip .4cm

We now explain our conditions {\bf (\ref{OpStruct}-i)} and {\bf (\ref{OpStruct}-ii)}:

\vskip .5cm

\subsection{\bf Condition {\bf (\ref{OpStruct}-i)}: good input categories.}
If $\mathcal{C}$ is a category containing a terminal object $\ast$,
then let us refer to a morphism
$f:B\longrightarrow C$
as an {\em acute quotient}
in $\mathcal{C}$
if $f$ appears in at least one acute, bicartesian square
$$\xymatrix{
A\ar@{^{(}->}[r]\ar[d]
&B\ar[d]^{f}\cr
\ast\ar@{^{(}->}[r]
&C
}$$
We will tend to denote acute quotients as epimorphisms:
$B\lra\!\!\!\!\rightarrow C$.
In the categories that we will be interested in,
acute quotients will turn out to be epimorphisms,
so this won't lead to any confusion.

Recall that an {\em image factorization}
of a given morphism
$f:B\longrightarrow C$ in $\mathcal{C}$
is any factorization
$B\rightarrow\non{Im}(f)\hookrightarrow C$
of $f$
through a monomorphism,
such that this factorization is {\em minimal}
in the sense that for any other factorization
$B\rightarrow C'\hookrightarrow C$
of $f$ through a monomorphism,
there exists a morphism $\mathrm{Im}(f)\lra C'$
for which the diagram
$$
\xy
(0,0)*+{B}="B";
(20,7)*+{\mathrm{Im}(f)}="Im";
(20,-7)*+{\ C'\ }="C0";
(40,0)*+{C}="C";
{\ar@{->} "B"; "Im"};
{\ar@{->} "B"; "C0"};
{\ar@{^{(}->} "Im"; "C"};
{\ar@{^{(}->} "C0"; "C"};
{\ar@{->} "Im"; "C0"};
\endxy
$$
commutes.
Commutativity of the above diagram implies that this morphism $\mathrm{Im}(f)\lra C'$ is unique, and is necessarily a monomorphism.
Thus when it exists, the image factorization of $f$ is unique up to unique isomorphism.

\vskip .4cm

\begin{definition}
{\bf (Good input category.)}
We say that a category
$\mathcal{C}$
is a {\em good input category} if it satisfies the following three axioms:

\vskip .2cm

\begin{itemize}
\item[{\bf GI-1}]
{\bf (Well-powered)}
The category
$\mathcal{C}$
is {\em well-powered},
meaning that each object $C$ in $\mathcal{C}$
has an essentially small poset
$\bold{Mon}_C\mathcal{C}$
of subobjects;

\vskip .2cm

\item[{\bf GI-2}]
{\bf (Nonempty point sets)}
The category
$\mathcal{C}$
contains a terminal object $\ast$,
and each object $C$ in $\mathcal{C}$ has a nonempty point set:
$\non{Hom}_{\mathcal{C}}(\ast,C)\not=\O$;

\vskip .2cm

\item[{\bf GI-3}]
{\bf (Second isomorphism)}
Every morphism $f:B\lra C$ in
$\mathcal{C}$
admits an image factorization.
If
$f:B\xhookrightarrow{\ \ \ } C'\lra\!\!\!\!\rightarrow C$
is the composite of a monomorphism followed by an acute quotient in
$\mathcal{C}$,
then $f$'s image factorization
$B\rightarrow\non{Im}(f)\hookrightarrow C$
consists of an acute quotient followed by a monomorphism.
\end{itemize}

\end{definition}

\vskip .4cm

Let us give several examples of good input categories. Part of our intention with these examples is to illustrate that good input categories are categories that come with a natural, and rather operadic-flavored calculus of "substituting objects for points."

\vskip .5cm

\begin{example}
{\bf (Groups.)}
The category
$\bold{Grp}$ of groups is a good input category.
Axiom {\bf GI-3} is nothing but the Second Isomorphism Theorem for groups.
Indeed, every acute quotient in $\bold{Grp}$
is isomorphic to a quotient $p:G\lra\!\!\!\!\rightarrow G/N$
by a normal subgroup $N$.
If $H$ is any subgroup of $G$,
then its image under $p$ is
$HN/N=\{hN:h\in H\}$,
and the Second Isomorphism Theorem realizes this image as the acute quotient $H\lra\!\!\!\!\rightarrow H/(H\cap N)$.
\end{example}

\vskip .2cm

\begin{example}
{\bf (Well-powered abelian categories.)}
Every well powered abelian category
$\mathcal{A}$ is a good input category.
Axiom {\bf GI-3}
follows from the facts that we can compute the image of any morphism $f$
in $\Ac$ as $\mathrm{Im}(f)=\mathrm{coker}(\mathrm{ker}(f))$
and that every monomorphism is the kernel of its cokernel.
\end{example}

\vskip .2cm

\begin{example}
{\bf (Various categories of nonempty sets.)}
The category
$\textsf{NE}\Sets$
of nonempty sets is a good input category, as is the category
$\fnS$
of nonempty finite sets.
In both these categories, acute quotients take the form
$I\lra\!\!\!\!\rightarrow I/J$,
where $J$ is a subset of $I$,
and $I/J$ denotes the set gotten by collapsing $J$ down to a point.
The verification of axiom {\bf GI-3} falls into two distinct cases,
namely composites $I'\xhookrightarrow{\ \ \ }I\lra\!\!\!\!\rightarrow I/J$
where the image of $I'$ intersects $J$, and those where it does not.

Note that neither the category $\bold{Sets}$ of sets
nor the category $\textsf{F}\bold{Sets}$
of finite sets is a good input category,
since the empty set $\O$ lies in both these categories,
but contains no points.
\end{example}

\vskip .2cm

\begin{example}
{\bf (Simplicial objects in a pointed, good input category.)}
If $\mathcal{C}$ is a good input category which is pointed,
then the category $\non{s}\mathcal{C}=[\bold{\Delta}\!^{\non{op}},\ \mathcal{C}]$
of simplicial objects in
$\mathcal{C}$
is a good input category.
Thus the category
$\non{s}\bold{Grp}$
of simplicial groups is a good input category,
as is the category
$\non{s}\mathcal{A}$
of simplicial objects in any well-powered abelian category
$\mathcal{A}$.

\end{example}

\vskip .2cm

\begin{example}
{\bf (Various categories of nonempty simplicial sets.)}
Each of the categories of simplicial objects
$\non{s}\textsf{NE}\Sets=[\bold{\Delta}\!^{\non{op}},\ \textsf{NE}\Sets]$
and
$\non{s}\fnS=[\bold{\Delta}\!^{\non{op}},\ \fnS]$
is a good input category.

Likewise, let
$\textsf{F}\non{s}\fnS$
denote the full subcategory of
$\non{s}\fnS$
consisting of nonempty, {\em finite-dimensional} simplicial sets.
A simplicial set
$X_\bullet$
is an object in
$\textsf{F}\non{s}\fnS$
if $X_\bullet$ is $m$-skeletal for some finite $m\geq0$,
and if $X_n\not=\O$ for each $n$.
Then this category
$\textsf{F}\non{s}\fnS$
is also a good input category.

\end{example}

\vskip .2cm

\begin{example}
{\bf (Little intervals.)}
For each integer $n\geq 1$,
define the {\em little} $n$-{\em interval}
(with endpoints)
to be the 1-skeletal simplicial set
$$I_{\bullet}(n)=\ \ \bullet\!\!\!\rightarrow\!\!\relbar\!\!\!\bullet\!\!\!\rightarrow\!\!\relbar\dotsm\rightarrow\!\!\relbar\!\!\!\bullet$$
containing $n$ many 0-simplices.
Notice that
$I_{\bullet}(1)$
consists of a single point $\bullet$.

Let
$\bold{Int}$
denote the full subcategory of
$\textsf{F}\non{s}\fnS$
consisting of the little $n$-intervals, for all
$n\geq 1$.
Then $\bold{Int}$
is a good input category.
This category $\bold{Int}$
bears a loose resemblance to the operad of little intervals or little 1-cubes,
in the sense that each acute bicartesian square in
$\bold{Int}$
describes a procedure in which we input some little interval
$I_{\bullet}(n)$
into another
$I_{\bullet}(m)$,
at a fixed 0-simplex
$x\in I_{0}(m)$,
to obtain the little interval
$I_{\bullet}(m+n-1)$.

\end{example}

\vskip .2cm

\begin{example}
{\bf (Simplicial rooted trees.)}
Let a {\em simplicial rooted tree}
be any 1-skeletal, 1-connected simplicial set
$X_\bullet$ in
$\textsf{F}\non{s}\fnS$
with the property that, at each {\em vertex}
$v\in X_0$,
there exists at most one nondegenerate 1-simplex
$e\in X_1$
of the form
$$\xy
{\ar^{\ \ \ \ \ \ \ e}
(0,0)*+{v\bullet\!\!\!};
(15,0)*+{\!\!\!\bullet}
**\dir{-}%
?<* \dir{-}%
?(.55)*%
\dir{>}
}%
\endxy
$$
We refer to nondegenerate 1-simplices in
$X_\bullet$ as {\em edges}.
Notice that, according to our definition,
the simplicial set
$\bullet=\Delta\!^0$
is a simplicial rooted tree with no edges.

The vertices of any simplicial rooted tree
$X_\bullet$
come with a natural partial ordering "$\preccurlyeq$,"
where
$v_{1}\preccurlyeq v_2$ whenever
$X_\bullet$ contains an oriented path running from
$v_{1}$ to $v_2$.
Let
$v_{\non{root}}$
denote the terminal vertex in this ordering,
and let an {\em input} be any initial vertex in this ordering.
At an arbitrary vertex
$v_{0}\in X_0$,
let
$\non{In}(v_{0})$
denote the set
of those inputs $v$
in this ordering for which
$v\preccurlyeq v_0$.

Let us say that a simplicial morphism
$f:X_{\bullet}\longrightarrow Y_\bullet$
of simplicial rooted trees is {\em full}
if
$$f(\non{In}(v_{\non{root}}))=\non{In}(f(v_{\non{root}})),$$
and let
$\bold{Tree}$
denote the subcategory of
$\textsf{F}\non{s}\fnS$
consisting of simplicial rooted trees and full simplicial morphisms.
Then $\bold{Tree}$
is a good input category,
and its acute bicartesian squares correspond to root-to-input graftings of trees.

\end{example}

\vskip .5cm

\subsection{Condition {\bf (\ref{OpStruct}-ii)}: the functor {\em F}}
If $\mathcal{C}$ and $\mathcal{D}$
are any two categories containing terminal objects
$\ast_\mathcal{C}$ and $\ast_\mathcal{D}$,
respectively,
then we call a functor
$F:\mathcal{C}\longrightarrow\mathcal{D}$
{\em right-acute}
if $F$ preserves acute cocartesian squares.
In other words, if
$$\xymatrix{
C_2\ar[r]\ar[d]
&C_3\ar[d]\cr
\ \ \ast_{\mathcal{C}}\ \ar@{^{(}->}[r]
&C_1
}$$
is any cocartesian square in
$\mathcal{C}$,
then its image under any right-acute functor $F:\mathcal{C}\longrightarrow\mathcal{D}$ is a cocartesian square
$$\xymatrix{
F(C_2)\ar[r]\ar[d]
&F(C_3)\ar[d]\cr
\ \ \ast_{\mathcal{D}}\ \ar@{^{(}->}[r]
&F(C_1)
}$$
in
$\mathcal{D}$.
In particular,
if $F$ is right-acute, then it preserves terminal objects:
$F(\ast_\mathcal{C})=\ast_\mathcal{D}$.

The functor $F$
is {\em acute}
if it preserves acute bicartesian squares.

\vskip .4cm

As for the locality conditions on $F$,
recall that for each object $C$ in $\Cc$,
we let $F_C$ denote the composite
$F_{C}:\bold{Mon}_{C}\Cc\xrightarrow{{}_{U_C}}\Cc\xrightarrow{{}_{\ F\ }}\Dc$.
From general nonsense, one knows that if
$\eta:F_{C}\Longrightarrow E$ is a natural transformation
between $\mathcal{D}$-valued functors on $\bold{Mon}_{C}\Cc$,
such that the component
$\eta_{B\hookrightarrow C}$
at each object $B\hookrightarrow C$ in $\bold{Mon}_{C}\Cc$
is an epimorphism
$$
\eta_{B\hookrightarrow C}:F(B)\lra\!\!\!\!\rightarrow E(B\!\hookrightarrow\!C)
$$
in $\mathcal{D}$, then
$\eta$ is an epimorphism in the functor category
$[\bold{Mon}_{C}\Cc,\mathcal{D}]$,
i.e., $\eta$ is an object in
the poset
$\bold{Epi}^{F_C}[\bold{Mon}_{C}\Cc,\mathcal{D}]$.

We say that $F$ has
{\em locally defined relations} over
$\mathcal{C}$
if, at each object $C$ in $\Cc$,
the objects in
$\bold{Epi}^{F_C}[\bold{Mon}_{C}\Cc,\mathcal{D}]$
are precisely those natural transformations
$\eta:F_{C}\Longrightarrow E$
whose every component is an epimorphism.

Another way to say this is that if
$F$ has locally defined relations over $\mathcal{C}$,
then the restriction of natural transformations in
$\bold{Epi}^{F_C}[\bold{Mon}_{C}\Cc,\mathcal{D}]$
to their components at any given object
$B\hookrightarrow C$ in $\bold{Mon}_{C}\Cc$
constitutes a functor
$$\non{res}_{B\hookrightarrow C}:\bold{Epi}^{F_C}[\bold{Mon}_{C}\Cc,\mathcal{D}]\longrightarrow\bold{Epi}^{F(B)}\mathcal{D}$$
We say that
$F$ has
{\em locally small relations} over $\Cc$
if, for each object
$B$ in $\mathcal{C}$,
the poset
$\bold{Epi}^{F(B)}\mathcal{D}$
is essentially small.

In examples, the fact that $F$ has locally defined and locally small relations over $\Cc$ tends to follow from general properties of $\Dc$.
Let us give several basic examples:

\vskip .5cm

\begin{example}
{\bf (Right-exact functors between abelian categories.)}

If both $\mathcal{C}$ and $\mathcal{D}$ are abelian categories,
then an additive functor $F:\mathcal{C}\longrightarrow\mathcal{D}$
is right-acute if and only it is right-exact,
and is acute if and only if it is exact.

Such a functor $F$ always has locally defined relations over
$\mathcal{C}$.
If the abelian category
$\mathcal{D}$
is well-powered,
i.e.,
if each poset of subobjects
$\bold{Mon}_{D}\mathcal{D}$
in $\mathcal{D}$ is essentially small,
then
$F$ has locally small relations over
$\mathcal{C}$
as well.

\end{example}

\vskip .2cm

\begin{example}
{\bf (Abelianization.)}
Let $\bold{Ab}$ be the category of abelian groups,
and let 
$$\non{ab}:\bold{Grp}\longrightarrow\bold{Ab}$$
be the abelianization functor:
$\non{ab}(G)=G/[G,G]$.
Then $\non{ab}$
is right-acute, and has locally defined and locally small relations over
$\bold{Grp}$.
\end{example}

\vskip .2cm

\begin{example}
{\bf (Connected components of simplicial sets.)}
Let
$\non{s}\textsf{NE}\Sets=[\bold{\Delta}\!^{\non{op}},\textsf{NE}\Sets]$
be the category of nonempty simplicial sets, and let
$$\pi_{0}:\non{s}\textsf{NE}\Sets\longrightarrow\Sets$$
be the functor that that returns the set
$\pi_{0}(X_\bullet)$
of connected components of each nonempty simplicial set
$X_\bullet$.
Then $\pi_{0}$
is right-acute, with locally defined and locally small relations over
$\non{s}\textsf{NE}\Sets$.
\end{example}

\vskip .5cm

We will give more examples below in Section \ref{q grass}.

We now begin proving Theorem \ref{FMop}.
The proof will use several facts about good input categories and right-acute functors,
which we prove in the following Lemmas \ref{gi1} through \ref{gi7}:

\vskip .7cm

\begin{lemma}\label{gi1}
If $\mathcal{C}$ is a good input category,
then every terminal morphism
$C\longrightarrow\ast$
in $\mathcal{C}$
is an epimorphism.
\end{lemma}

\begin{proof}
Since $\mathcal{C}$
has nonempty point sets, each terminal morphism
$p:C\longrightarrow\ast$
admits at least one section
$s:\ast\longrightarrow C$,
verifying that $p$
is an epimorphism.
\end{proof}

\vskip .2cm

\begin{lemma}\label{gi2}
If $\mathcal{C}$ is a good input category,
then every acute quotient
$p:B\longrightarrow C$
in $\mathcal{C}$
is an epimorphism.
\end{lemma}

\begin{proof}
Choose any bicartesian square in
$\mathcal{C}$
that realizes $p$ as an acute quotient:
$$\xymatrix{
{B_0}\ar@{^{(}->}[r]
\ar@{->}[d]
&B\ar@{->}[d]^{p}\\
\ast\ar@{^{(}->}[r]
&C
}$$
By Lemma \ref{gi1},
the morphism
$B_{0}\longrightarrow\ast$
appearing in this square is an epimorphism.
Since the square is cocartesian,
this implies that $p$ is an epimorphism.
\end{proof}

\vskip .2cm

\begin{lemma}\label{gi3}
Every terminal morphism $p:C\longrightarrow\ast$
in a good input category $\mathcal{C}$
is its own image factorization, meaning that any image factorization of $p$
is of the form
$$p:\ 
\xymatrix{
C\ar@{->}[r]
&{\non{Im}(p)}\ar@{->}[r]^{\ \ \ \sim\ }
&{\ast}
}
$$
with the morphism at right being an isomorphism.
\end{lemma}

\begin{proof}
The square
$$
\xymatrix{
C\ar@{->}[r]^{\non{id}}\ar[d]_{p}
&C\ar[d]^{p}
\\
{\ast}\ar@{->}[r]
&{\ast}
}
$$
realizes $p$ as an acute quotient in
$\mathcal{C}$.
Thus by axiom {\bf GI-3},
the morphism
$p$ has an image factorization
$$\xymatrix{C\ar@{->>}[r]^{\!\!\!q\ }&{\non{Im}(p)}\ar@{^{(}->}[r]^{\ \ \ r}&{\ast}}$$
wherein the morphism $q$ is an acute quotient.
By Lemma \ref{gi2}, $q$ is indeed an epimorphism.
Because $\mathcal{C}$ has nonempty point sets,
we can choose a point $s$ in $\non{Im}(p)$:
$$
\xy
(0,0)*+{C}="c";
(40,0)*+{\ast}="*";
(20,0)*+{\non{Im}(p)}="im";
{\ar@{->>}^{q\ \ } "c"; "im"};
{\ar@{^{(}->}_{\ \ r} "im"; "*"};
{\ar@{-->}@/_{1.25pc}/_{s} "*"; "im"};
\endxy
$$
Since $r{}_{{}^{^\circ}}s=\non{id}$,
we have
$r{}_{{}^{^\circ}}s{}_{{}^{^\circ}}r{}_{{}^{^\circ}}q=r{}_{{}^{^\circ}}q$.
But because $r$ is a monomorphism and $q$ is an epimorphism,
this last identity implies that we also have $s{}_{{}^{^\circ}}r=\non{id}$.
\end{proof}

\vskip .2cm

\begin{lemma}\label{gi4}
If $\mathcal{C}$ is a good input category,
then every composite in
$\mathcal{C}$
of the form
$$f:\ B\xrightarrow{\ \ p\ \ }\ast\xrightarrow{\ \ c\ \ }C$$
is an image factorization of $f$.
\end{lemma}

\begin{proof}
The universal property of image factorizations dictates that the image factorization of any composite of the form
$$
\xy
(0,0)*+{B}="b";
(40,0)*+{C}="c";
(20,0)*+{C_0}="c0";
{\ar@{->}^{p\ \ } "b"; "c0"};
{\ar@{^{(}->}^{\ \ c} "c0"; "c"};
\endxy
$$
is the factorization
$
\xy
(0,0)*+{B}="b";
(40,0)*+{C}="c";
(20,0)*+{\non{Im}(p)}="im";
{\ar@{->} "b"; "im"};
{\ar@{^{(}->}^{\ \ c{}_{{}^{^\circ}}p_0} "im"; "c"};
\endxy
$,
where
$
\xy
(0,0)*+{B}="b";
(40,0)*+{C_0}="c";
(20,0)*+{\non{Im}(p)}="im";
{\ar@{->} "b"; "im"};
{\ar@{^{(}->}^{\ \ p_0} "im"; "c"};
\endxy
$
is $p$'s image factorization.
In the case where
$\mathcal{C}$ is a good input category and
$C_{0}=\ast$,
the present Lemma then follows from Lemma \ref{gi3}.
\end{proof}

\vskip .2cm

\begin{lemma}\label{gi5}
Suppose given a commutative diagram of the form
$$
\xymatrix{
A\ar@{->}[r]^{f}\ar[d]_{a}
&B\ar[d]^{g}
\\
{\ast}\ar@{->}[r]
&{C}
}
$$
in a good input category
$\mathcal{C}$.
Then there exists a point
$x:\ast\dashrightarrow{\non{Im}(g)}$
for which the diagram
$$
\xymatrix{
A\ar@{->}[r]^{f}\ar[d]_{a}
&B\ar[d]
\\
{\ast}\ar@{-->}[r]_{x\ \ \ }
&{\non{Im}(g)}
}
$$
commutes.
\end{lemma}

\begin{proof}
By Lemma \ref{gi4}, we know that the composite
$A\xrightarrow{\ \ a\ \ }\ast\longrightarrow C$
is its own image factorization.
Thus, according to the universal property of image factorizations,
it must factor through the composite
$A\longrightarrow B\longrightarrow\non{Im}(g)\hookrightarrow C$.
\end{proof}

\vskip .2cm

To setup the next Lemma, let
$\mathcal{C}$
be a good input category,
and let
$$
\xy
(0,0)*+{B\ }="b";
(15,0)*+{C}="c";
(30,0)*+{D_2}="d1";
(45,0)*+{D_1}="d2";
{\ar@{^{(}->}^{i} "b"; "c"};
{\ar@{->>}^{p} "c"; "d1"};
{\ar@{->}^{q} "d1"; "d2"}
\endxy
$$
be a composite in
$\mathcal{C}$
consisting of a monomorphism $i$
followed by an acute quotient $p$
followed by an arbitrary morphism $q$.
There are two, a priori distinct ways to obtain what our intuition tells us should be the
"image" of the composite
$q{}_{{}^{^\circ}}p{}_{{}^{^\circ}}i:B\longrightarrow D_1$.
We can simply take the image of $B$ in $D_1$,
or we can first take the image of $B$ in $D_2$,
and then take the image of this image in $D_1$.
These two procedures are described by the upper and lower paths in the commutative diagram
$$
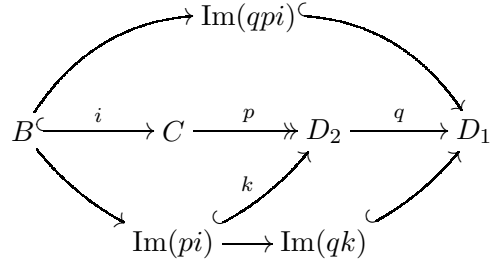
\xy
(0,0)*+{B}="b";
(20,0)*+{C}="c";
(40,0)*+{D_2}="d1";
(60,0)*+{D_1}="d2";
(30,15)*+{\non{Im}(qpi)}="im1";
(20,-15)*+{\non{Im}(pi)}="im2";
(40,-15)*+{\non{Im}(qk)}="im3";
{\ar@{^{(}->}^{i} "b"; "c"};
{\ar@{->>}^{p} "c"; "d1"};
{\ar@{->}^{q} "d1"; "d2"};
{\ar@{->}@/^{1.35pc}/ "b"; "im1"};
{\ar@{^{(}->}@/^{1.35pc}/ "im1"; "d2"};
{\ar@{->}@/_{.5pc}/ "b"; "im2"};
{\ar@{->} "im2"; "im3"};
{\ar@{^{(}->}@/_{.5pc}/^{k} "im2"; "d1"};
{\ar@{^{(}->}@/_{.5pc}/ "im3"; "d2"}
\endxy
$$
The next Lemma says that these two ways of computing the image give the same answer in any good input category $\Cc$.

\vskip .2cm

\begin{lemma}\label{gi6}
In any good input category $\Cc$,
the factorization
$
\xy
(0,0)*+{B}="b";
(17.5,0)*+{\non{Im}(qk)}="c";
(35,0)*+{D_1}="d1";
{\ar@{->} "b"; "c"};
{\ar@{^{(}->} "c"; "d1"}
\endxy
$
of the composite
$q{}_{{}^{^\circ}}p{}_{{}^{^\circ}}i:B\longrightarrow D_1$,
implicit in the lower path of the above diagram,
is an image factorization of
$q{}_{{}^{^\circ}}p{}_{{}^{^\circ}}i$.
\end{lemma}

\begin{proof}
It suffices to provide a morphism
$\non{Im}(qk)\dashrightarrow\non{Im}(qpi)$
that makes the diagram
$$
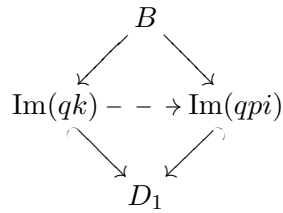
\xy
(0,0)*+{\non{Im}(qk)}="im1";
(24,0)*+{\non{Im}(qpi)}="im2";
(12,12)*+{B}="b";
(12,-12)*+{D_1}="d";
{\ar@{->} "b"; "im1"};
{\ar@{->} "b"; "im2"};
{\ar@{-->} "im1"; "im2"};
{\ar@{_{(}->} "im1"; "d"};
{\ar@{^{(}->} "im2"; "d"};
\endxy
$$
commute.
The existence of such a morphism will follow from the universal property of
$\non{Im}(qk)$
if we can show that there exists a morphism
$\non{Im}(pi)\dashrightarrow\non{Im}(qpi)$
for which the diagram$$
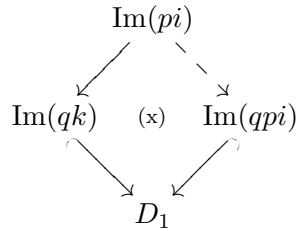
\xy
(0,0)*+{\non{Im}(qk)}="im1";
(26,0)*+{\non{Im}(qpi)}="im2";
(13,13)*+{\non{Im}(pi)}="b";
(13,-13)*+{D_1}="d";
(13,0)*+{{}_\mathrm{(x)}};
{\ar@{->} "b"; "im1"};
{\ar@{-->} "b"; "im2"};
{\ar@{_{(}->} "im1"; "d"};
{\ar@{^{(}->} "im2"; "d"};
\endxy
$$
commutes.

To this end, note that since
$C\longrightarrow\!\!\!\!\rightarrow D_2$
is an acute quotient,
axiom {\bf GI-3} tells us that the morphism
$B\longrightarrow \non{Im}(pi)$
is an acute quotient.
Thus by Lemma \ref{gi2},
this last morphism is an epimorphism
$B\longrightarrow\!\!\!\!\rightarrow\non{Im}(pi)$.
Let
$$
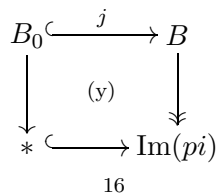
\xy
(0,0)*+{B_0}="B0";
(20,0)*+{B}="B";
(0,-15)*+{\ \ast\ }="*";
(20,-15)*+{\mathrm{Im}(pi)}="Im";
(10,-7.5)*+{{}_\mathrm{(y)}};
{\ar@{^{(}->}^{j} "B0"; "B"};
{\ar@{->} "B0"; "*"};
{\ar@{^{(}->} "*"; "Im"};
{\ar@{->>} "B"; "Im"}
\endxy
$$
be a bicartesian square in
$\mathcal{C}$
that realizes
$B\longrightarrow\!\!\!\!\rightarrow\non{Im}(pi)$
as an acute quotient.
Then the morphism
$qk:\non{Im}(pi)\longrightarrow D_1$
gives us a commutative square
$$\xymatrix{
{B_0}\ar@{^{(}->}[r]^{j}
\ar@{->}[d]
&B\ar@{->}[d]^{qpi}\\
\ast\ar@{^{(}->}[r]
&{D_1}
}$$
Thus by Lemma \ref{gi5},
we get a commutative square
$$\xymatrix{
{B_0}\ar@{^{(}->}[r]^{j}
\ar@{->}[d]
&B\ar@{->}[d]\\
\ast\ar@{^{(}-->}[r]
&{\non{Im}(qpi)}
}$$
Since the square (y) is cocartesian,
this gives rise to the morphism
$\non{Im}(pi)\dashrightarrow\non{Im}(qpi)$
we're after.
Commutativity of the diagram (x) follows from the fact that
$B\longrightarrow\!\!\!\!\rightarrow\non{Im}(pi)$
is an epimorphism.
\end{proof}

\vskip .2cm

\begin{lemma}\label{gi7}
If $\mathcal{C}$ is a good input category
and $F:\mathcal{C}\longrightarrow\mathcal{D}$
is a right-acute functor,
then $F$ takes acute quotients in
$\mathcal{C}$
to epimorphisms in
$\mathcal{D}$.
\end{lemma}

\begin{proof}
Let $\ast_\mathcal{C}$ and $\ast_\mathcal{D}$
denote the terminal objects in
$\mathcal{C}$ and $\mathcal{D}$.
Let
$p:B\longrightarrow C$
be an acute quotient in
$\mathcal{C}$,
realized by an acute bicartesian square
$$\xymatrix{
{B_0}\ar@{^{(}->}[r]
\ar@{->}[d]
&B\ar@{->}[d]^{p}\\
\ast_{\mathcal{C}}\ar@{^{(}->}[r]
&C
}$$
in $\mathcal{C}$.
Choose a point
$b_{0}:\ast_{\mathcal{C}}\longrightarrow B_0$.
Since $F$ is right-acute, it takes the point $b_0$ to a point
$F(b_{0}):\ast_{\mathcal{D}}\longrightarrow F(B_0)$.
The existence of this point $F(b_{0})$
verifies that $F(B_0)\longrightarrow\ast_{\mathcal{D}}$
is an epimorphism.
Since $F$ takes the above bicartesian square to an acute cocartesian square in
$\mathcal{D}$,
this implies that
$F(p):F(B)\longrightarrow F(C)$
is an epimorphism.
\end{proof}

\vskip .5cm

We are now ready to prove Theorem \ref{FMop}:

\vskip .7cm

\begin{proof}
{\bf (Proof of Theorem \ref{FMop}.)}
First, note that since
$\mathcal{C}$ is well-powered,
while $F$ has locally small relations over
$\mathcal{C}$,
each category of relations
$\bold{Epi}^{F_C}[\bold{Mon}_{C}\mathcal{C},\mathcal{D}]$
is essentially small.
Thus the Fulton-MacPherson construction
produces a well defined functor
$\mathcal{F\!M}_{F}:\mathcal{C}_{\non{iso}}^{\non{op}}\longrightarrow\Sets$.
This proves part {\bf (I)} of Theorem \ref{FMop}.

To construct the composition assignment $\gamma$
whose existence we assert in part {\bf (II)} of Theorem \ref{FMop},
consider an arbitrary, acute bicartesian square
$$\xymatrix{
\ar@{}[dr]|{\non{(a)}}
C_2\ar@{^{(}->}[r]
\ar@{->>}[d]
&C_3\ar@{->>}[d]^{p}\\
\ast\ar@{^{(}->}[r]
&C_1
}$$
in
$\mathcal{C}$.
We define the map
$$\gamma_{\non{(a)}}:\mathcal{F\!M}_{F}(C_1)\times\mathcal{F\!M}_{F}(C_2)\longrightarrow\mathcal{F\!M}_{F}(C_3)$$
associated to this square as follows:

Recall that for each object $C$ in $\Cc$, we use $F_C$
to denote the composite $\bold{Mon}_{C}\Cc\xrightarrow{U_C}\Cc\xrightarrow{F}\Dc$.
Then consider an arbitrary pair of epimorphisms of
$\mathcal{D}$-valued functors
$$\eta_{1}:F_{C_1}\Longrightarrow\!\!\!\!\Rightarrow E_1{\non{\ \ \ \ \ \ over\ \ \ \ \ \ }}\bold{Mon}_{C_1}\mathcal{C}$$
$$\eta_{2}:F_{C_2}\Longrightarrow\!\!\!\!\Rightarrow E_2{\non{\ \ \ \ \ \ over\ \ \ \ \ \ }}\bold{Mon}_{C_2}\mathcal{C}$$
representing a pair in
$\mathcal{F\!M}_{F}(C_1)\times\mathcal{F\!M}_{F}(C_1)$.
The class
$\gamma_{\non{(a)}}(\eta_{1},\eta_{2})$
in
$\mathcal{F\!M}_{F}(C_3)$ 
is represented by a new epimorphism
$$\eta_{3}:F_{C_3}\Longrightarrow\!\!\!\!\Rightarrow E_3{\non{\ \ \ \ \ \ over\ \ \ \ \ \ }}\bold{Mon}_{C_3}\mathcal{C},$$
which we define piecewise over
$\bold{Mon}_{C_3}\mathcal{C}$,
according to the following two cases:

\vskip .2cm

{\bf Case-1:} If $B\hookrightarrow C_3$
is an object in
$\bold{Mon}_{C_3}\mathcal{C}$
that factors through the monomorphism
$C_{2}\hookrightarrow C_3$,
then define
$E_{3}(B\hookrightarrow C_3)$
to be the object
$E_{2}(B\hookrightarrow C_2)$
in $\mathcal{D}$.
As for the component of $\eta_3$ at
$B\hookrightarrow C_3$,
let it simply be the component of
$\eta_2$ at
$B\hookrightarrow C_2$:
$$\xy
{\ar^{\ \ \ \ \ \eta_{2,B\hookrightarrow C_2}}
(0,0)*+{F(B)};
(35,0)*+{E_{2}(B\hookrightarrow C_2)}
**\dir{-}%
?<* \dir{-}%
?>* \dir{>>}%
}%
\endxy
$$
The fact that this component is an epimorphism in $\Dc$
follows from our assumption that $F$
has locally defined relations over $\Cc$.

{\bf Case-2:} If $B\hookrightarrow C_3$
is an object in
$\bold{Mon}_{C_3}\mathcal{C}$
that does not factor through the monomorphism
$C_{2}\hookrightarrow C_3$,
then consider the composite
$$p|_{B}:
\xy
{\ar^{}
(0,0)*+{B\ };
(10,0)*+{\!\!}
**\dir{-}%
?<* \dir^{(}%
?>* \dir{>}%
}%
\endxy
\xy
{\ar^{\ \ \ p}
(0,0)*+{C_3};
(15,0)*+{C_1}
**\dir{-}%
?<* \dir{-}%
?>* \dir{>>}%
}%
\endxy
$$
in $\mathcal{C}$.
Since $\mathcal{C}$
is a good input category, axiom {\bf GI-3} tells us that this composite has an image factorization
$$\xy
{\ar^{}
(0,0)*+{B};
(10,0)*+{\!\!}
**\dir{-}%
?<* \dir{-}%
?>* \dir{>>}%
}%
\endxy
\xy
{\ar^{}
(0,0)*+{\non{Im}(p|_B)};
(18,0)*+{C_1}
**\dir{-}%
?<* \dir^{(}%
?>* \dir{>}%
}%
\endxy
$$
consisting of an acute quotient followed by a monomorphism.
The monomorphism
$\non{Im}(p|_B)\hookrightarrow C_1$
in this factorization is an object in
$\bold{Mon}_{C_1}\mathcal{C}$.
Define
$E_{3}(B\hookrightarrow C_3)$
to be the object
$E_{1}(\non{Im}(p|_B)\hookrightarrow C_1)$
in $\mathcal{D}$.

As for the component of
$\eta_3$ at $B\hookrightarrow C_3$,
note that since
$F$ is right-acute,
Lemma \ref{gi7} tells us that
$F$ takes the acute quotient
$B\longrightarrow\!\!\!\!\rightarrow{\non{Im}(p|_B)}$
to an epimorphism
$F(B)\longrightarrow\!\!\!\!\rightarrow F({\non{Im}(p|_B)})$
in $\mathcal{D}$.
Define the component of $\eta_3$ at
$B\hookrightarrow C_3$ to be the composite
$$\xy
{\ar^{\ \ \ \ \ F(B\longrightarrow\!\!\!\!\rightarrow{\non{Im}(p|_B)})}
(0,0)*+{F(B)};
(35,0)*+{}
**\dir{-}%
?<* \dir{-}%
?>* \dir{>>}%
}%
\endxy
\xy
{\ar^{\ \ \ \ \ \ \ \ \eta_{1,{\non{Im}(p|_B)}\hookrightarrow C_1}}
(0,0)*+{\!\!\!F\big(\non{Im}(p|_B)\big)};
(55,0)*+{E_{1}\big({\non{Im}(p|_B)}\hookrightarrow C_1\big)}
**\dir{-}%
?<* \dir{-}%
?>* \dir{>>}%
}%
\endxy
$$
The fact that the component at right is an epimorphism in $\Dc$
follows from our assumption that $F$
has locally defined relations over $\Cc$.
Hence this composite is an epimorphism in $\Dc$.

\vskip .2cm

{\bf Functoriality\ and\ naturality:}
To see that these components
$$\eta_{3,B\hookrightarrow C_3}:F(B)\longrightarrow\!\!\!\!\rightarrow E_{3}(B\hookrightarrow C_3)$$
constitute a natural transformation of functors over
$\bold{Mon}_{C_3}\mathcal{C}$,
it suffices to consider the situation of a morphism
$$
\xy
(0,0)*+{C_3}="C";
(-10,12)*+{\ A\ }="A";
(10,12)*+{B}="B";
{\ar@{_{(}->} "A"; "C"};
{\ar@{^{(}->}^{f} "A"; "B"};
{\ar@{^{(}->} "B"; "C"}
\endxy
$$
in
$\bold{Mon}_{C_3}\mathcal{C}$,
such that
$A\hookrightarrow C_3$
factors through
$C_2\hookrightarrow C_3$,
but
$B\hookrightarrow C_3$
does not factor through
$C_2\hookrightarrow C_3$.

In this situation, we have a commutative square
$$\xymatrix{
{A}\ar@{^{(}->}[r]
\ar@{->}[d]
&B\ar@{->}[d]^{p|_B}\\
\ast\ar@{^{(}->}[r]
&C_{3}
}$$
in $\mathcal{C}$.
Thus by Lemma \ref{gi5}, we know that there exists a point
$x:\ast\dashrightarrow{\non{Im}(p|_B)}$ for which the square
$$\xymatrix{
{A}\ar@{^{(}->}[r]
\ar@{->>}[d]
&B\ar@{->>}[d]\\
\ast\ar@{-->}[r]_{x\ \ \ \ }
&{\non{Im}(p|_B)}
}$$
commutes.
This square will be neither cartesian nor cocartesian in general.
Still, because $F$ is right acute,
specifically because $F$ preserves terminal objects, 
the image of this last square under $F$ will be the commutative square (x) in the diagram
$$
\xy
(0,0)*+{F(A)}="fa";
(0,30)*+{F(B)}="fb";
(25,30)*+{F\big(\non{Im}(p|_B)\big)}="fim";
(63,30)*+{E_{1}\big(\non{Im}(p|_B)\hookrightarrow C_1\big)}="eim";
(63,0)*+{E_{2}(A\hookrightarrow C_1)}="ea";
(25,15)*+{\ \ast_{\mathcal{D}}}="*";
(13,19)*+{\non{{}_{(x)}}}="x";
(25,5)*+{\non{{}_{(y)}}};
{\ar@{->}^{F(f)} "fa"; "fb"};
{\ar@{->>} "fa"; "ea"};
{\ar@{->>} "fb"; "fim"};
{\ar@{->>}^{\eta_1\ \ \ \ \ } "fim"; "eim"};
{\ar@{->} "fa"; "*"};
{\ar@{->}_{F(x)} "*"; "fim"};
{\ar@{-->}_{t} "ea"; "*"};
{\ar@{-->}_{\eta_{1}{}_{{}^{^\circ}}F(x){}_{{}^{^\circ}}t} "ea"; "eim"}
\endxy
$$
The terminal morphism $t$ appearing in this diagram makes the triangle (y) commute.
Thus if we define
$$
E_{3}(f):\ \ E_{3}(A\!\hookrightarrow\!C_3)\ \lra\ E_{3}(B\!\hookrightarrow\!C_3)
$$
to be the composite $\eta_{1}{}_{{}^{^\circ}}F(x){}_{{}^{^\circ}}t$
appearing in the above diagram,
then the outer square in the diagram commutes, giving us naturality for $f$. 

Checking that these morphisms $E(f)$
fit together with the functorial structure of
$E_1$ and $E_2$ to produce a functorial structure on $E_3$
is a straightforward matter.
We leave it to the reader.

\vskip .2cm

{\bf Operadic axioms:} It remains to verify axioms {\bf OP-1} through {\bf OP-3} for our assignment
$\gamma$.

The verification of {\bf OP-1} follows from an inspection of
$\mathcal{F}\!\mathcal{M}_F$'s
functorial structure and our definition of the composition
$\gamma$
on
$\mathcal{F}\!\mathcal{M}_F$.
We leave the the details to the reader.

To verify {\bf OP-2}, suppose given a diagram
$$\xymatrix{
C_3\ar[r]\ar[d]
&C_{23}\ar[r]\ar[d]
&C_{123}\ar[d]\\
\ast\ar@{^{(}->}[r]
&C_2\ar[r]\ar[d]
&C_{12}\ar[d]\\
\ &\ast\ar@{^{(}->}[r]
&C_1
}$$
in $\mathcal{C}$,
whose every square is bicartesian.
Fix an object $B\hookrightarrow C_{123}$ in
$\bold{Mon}_{C_{123}}\mathcal{C}$,
and consider a triad of epimorphisms
$$\eta_{1}:F_{C_1}\Longrightarrow\!\!\!\!\Rightarrow E_1{\non{\ \ \ \ \ \ over\ \ \ \ \ \ }}\bold{Mon}_{C_1}\mathcal{C}$$
$$\eta_{2}:F_{C_2}\Longrightarrow\!\!\!\!\Rightarrow E_2{\non{\ \ \ \ \ \ over\ \ \ \ \ \ }}\bold{Mon}_{C_2}\mathcal{C}$$
$$\eta_{3}:F_{C_3}\Longrightarrow\!\!\!\!\Rightarrow E_3{\non{\ \ \ \ \ \ over\ \ \ \ \ \ }}\bold{Mon}_{C_3}\mathcal{C},$$
representing an element in
$\FM{F}(C_1)\times\FM{F}(C_2)\times\FM{F}(C_3)$.

We have to consider three cases:

\vskip .2cm

{\bf Case-1:} If our monomorphism
$B\hookrightarrow C_{123}$
factors through
$C_{3}$,
then consider the commutative diagrams
$$
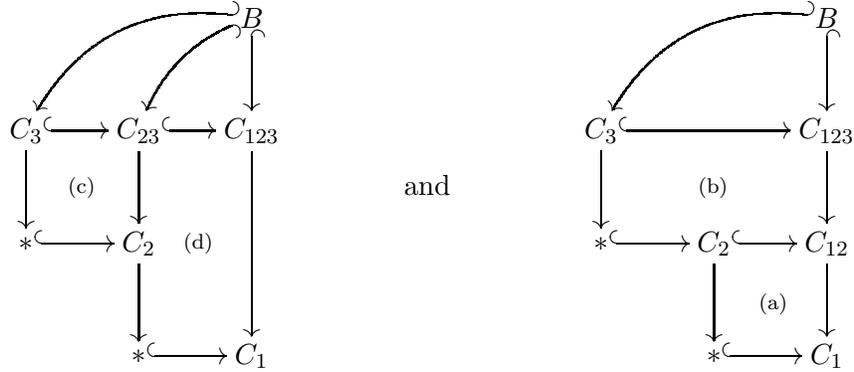
\vcenter{\vbox{
\xy
(0,0)*+{C_3}="c3";
(15,0)*+{C_{23}}="c23";
(30,0)*+{C_{123}}="c123";
(0,-15)*+{\ast}="*23";
(15,-15)*+{C_2}="c2";
(15,-30)*+{\ast}="*12";
(30,-30)*+{C_1}="c1";
(30,15)*+{B}="b";
(7.5,-7.5)*+{\non{{}_{(c)}}};
(23,-15)*+{\non{{}_{(d)}}};
{\ar@{^{(}->} "b"; "c123"};
{\ar@{->} "c3"; "*23"};
{\ar@{^{(}->} "*23"; "c2"};
{\ar@{->} "c2"; "*12"};
{\ar@{^{(}->} "*12"; "c1"};
{\ar@{^{(}->} "c3"; "c23"};
{\ar@{^{(}->} "c23"; "c123"};
{\ar@{->} "c23"; "c2"};
{\ar@{->} "c123"; "c1"};
{\ar@{_{(}->}@/_{1.5pc}/ "b"; "c3"};
{\ar@{^{(}->}@/_{.65pc}/ "b"; "c23"}
\endxy}}
\ \ \ \ \ \ \ \ \ \ \ \ \mathrm{and}\ \ \ \ \ \ \ \ \ \ \ \ \ 
\vcenter{\vbox{
\xy
(0,0)*+{C_3}="c3";
(30,0)*+{C_{123}}="c123";
(0,-15)*+{\ast}="*23";
(15,-15)*+{C_2}="c2";
(30,-15)*+{C_{12}}="c12";
(15,-30)*+{\ast}="*12";
(30,-30)*+{C_1}="c1";
(30,15)*+{B}="b";
(23,-23)*+{\non{{}_{(a)}}};
(15,-7.5)*+{\non{{}_{(b)}}};
{\ar@{^{(}->} "b"; "c123"};
{\ar@{->} "c3"; "*23"};
{\ar@{^{(}->} "*23"; "c2"};
{\ar@{->} "c2"; "*12"};
{\ar@{^{(}->} "*12"; "c1"};
{\ar@{^{(}->} "c3"; "c123"};
{\ar@{^{(}->} "c2"; "c12"};
{\ar@{->} "c123"; "c12"};
{\ar@{->} "c12"; "c1"};
{\ar@{_{(}->}@/_{1.5pc}/ "b"; "c3"}
\endxy}}
$$
In computing the double-composition associated to the diagram at left,
the fact that
$B\hookrightarrow C_{123}$
factors through
$C_{23}\hookrightarrow C_{123}$
means that in order to evaluate the composition associated to the square (d),
it suffices to evaluate the composition associate to the square (c)
at the object
$B\hookrightarrow C_{23}$ in
$\bold{Mon}_{C_{23}}\mathcal{C}$.
But since
$B\hookrightarrow C_{23}$
factors through
$C_{3}\hookrightarrow C_{23}$,
this evaluation is just the evaluation of the epimorphism
$\eta_{3}:F_{C_3}\Longrightarrow\!\!\!\!\Rightarrow E_3$
at the object
$B\hookrightarrow C_3$
in
$\bold{Mon}_{C_3}\mathcal{C}$.
Following similar reasoning,
we see that the double-composition associated to the diagram at right above also returns the evaluation of the epimorphism
$\eta_{3}:F_{C_3}\Longrightarrow\!\!\!\!\Rightarrow E_3$
at the object
$B\hookrightarrow C_3$
in
$\bold{Mon}_{C_3}\mathcal{C}$.

{\bf Case-2:} Suppose that our monomorphism
$B\hookrightarrow C_{123}$
factors through
$C_{23}\hookrightarrow C_{123}$,
but does not factor through
$C_{3}\hookrightarrow C_{23}$.
Then consider the commutative diagrams
$$
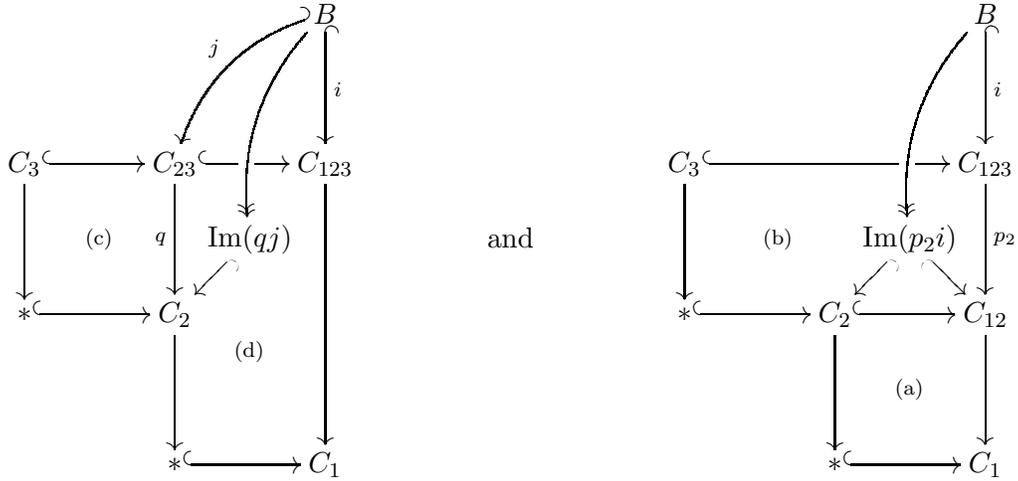
\vcenter{\vbox{
\xy
(0,0)*+{C_3}="c3";
(20,0)*+{C_{23}}="c23";
(40,0)*+{C_{123}}="c123";
(0,-20)*+{\ast}="*23";
(20,-20)*+{C_2}="c2";
(20,-40)*+{\ast}="*12";
(40,-40)*+{C_1}="c1";
(40,20)*+{B}="b";
(30,-10)*+{\non{Im}(qj)}="im";
(10,-10)*+{\non{{}_{(c)}}};
(30,-25)*+{\non{{}_{(d)}}};
{\ar@{^{(}->}^{i} "b"; "c123"};
{\ar@{->} "c3"; "*23"};
{\ar@{^{(}->} "*23"; "c2"};
{\ar@{->} "c2"; "*12"};
{\ar@{^{(}->} "*12"; "c1"};
{\ar@{^{(}->} "c3"; "c23"};
{\ar@{^{(}->}|(.475){\hole} "c23"; "c123"};
{\ar@{->}_{q} "c23"; "c2"};
{\ar@{->} "c123"; "c1"};
{\ar@{_{(}->}@/_{1pc}/_{j} "b"; "c23"};
{\ar@{->>}@/_{1pc}/ "b"; "im"};
{\ar@{^{(}->} "im"; "c2"};
\endxy}}
\ \ \ \ \ \ \ \ \ \ \ \ \ \mathrm{and}\ \ \ \ \ \ \ \ \ \ \ \ \ 
\vcenter{\vbox{
\xy
(0,0)*+{C_3}="c3";
(40,0)*+{C_{123}}="c123";
(0,-20)*+{\ast}="*23";
(20,-20)*+{C_2}="c2";
(40,-20)*+{C_{12}}="c12";
(20,-40)*+{\ast}="*12";
(40,-40)*+{C_1}="c1";
(40,20)*+{B}="b";
(30,-10)*+{\non{Im}(p_{2}i)}="im";
(30,-30)*+{\non{{}_{(a)}}};
(12.5,-10)*+{\non{{}_{(b)}}};
{\ar@{^{(}->}^{i} "b"; "c123"};
{\ar@{->} "c3"; "*23"};
{\ar@{^{(}->} "*23"; "c2"};
{\ar@{->} "c2"; "*12"};
{\ar@{^{(}->} "*12"; "c1"};
{\ar@{^{(}->}|(.74){\hole} "c3"; "c123"};
{\ar@{^{(}->} "c2"; "c12"};
{\ar@{->}^{p_2} "c123"; "c12"};
{\ar@{->} "c12"; "c1"};
{\ar@{->>}@/_{1pc}/ "b"; "im"};
{\ar@{_{(}->} "im"; "c12"};
{\ar@{^{(}->} "im"; "c2"}
\endxy}}
$$
where the existence of the monomorphism
$\non{Im}(p_{2}i)\hookrightarrow C_2$
appearing in the diagram at right
follows from the the existence of $j$ in the diagram at left, combined with the universal property of image factorizations.
When we compute the double-composition associated to the diagram at left,
we get the evaluation of the epimorphism
$\eta_{2}:F_{C_2}\Longrightarrow\!\!\!\!\Rightarrow E_2$
at the object
$\non{Im}(qj)\hookrightarrow C_2$
in
$\bold{Mon}_{C_2}\mathcal{C}$.
When we compute the double-composition associated to the diagram at right,
we get the evaluation of this same epimorphism
$\eta_{2}:F_{C_2}\Longrightarrow\!\!\!\!\Rightarrow E_2$,
but now at the object
$\non{Im}(p_{2}i)\hookrightarrow C_2$
in
$\bold{Mon}_{C_2}\mathcal{C}$.
But the universal property of image factorizations implies that
$\non{Im}(qj)\hookrightarrow C_2$
and
$\non{Im}(p_{2}i)\hookrightarrow C_2$
are isomorphic in
$\bold{Mon}_{C_2}\mathcal{C}$.
Thus these two answers are identical at the level of isomorphism classes of natural quotients of
$F_{C_2}$.

{\bf Case-3:} Suppose that our monomorphism
$B\hookrightarrow C_{123}$
doesn't even factor through
$C_{23}\hookrightarrow C_{123}$.
Consider then the commutative diagrams
$$
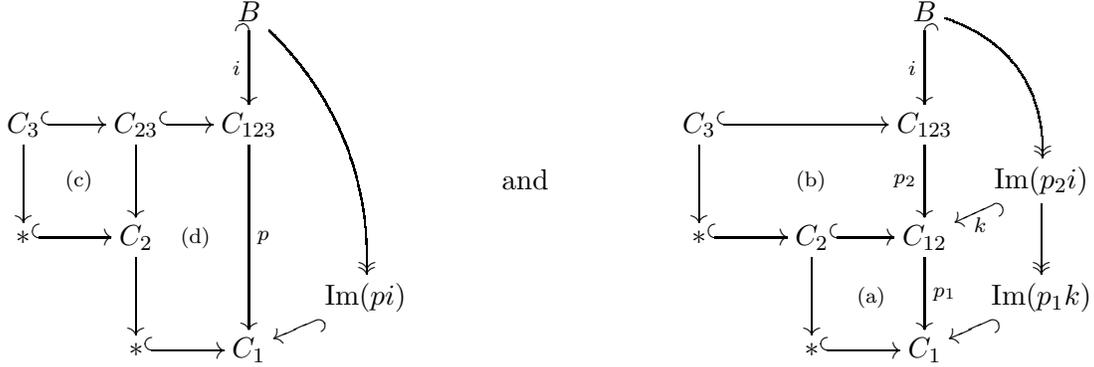
\vcenter{\vbox{
\xy
(0,0)*+{C_3}="c3";
(15,0)*+{C_{23}}="c23";
(30,0)*+{C_{123}}="c123";
(0,-15)*+{\ast}="*23";
(15,-15)*+{C_2}="c2";
(15,-30)*+{\ast}="*12";
(30,-30)*+{C_1}="c1";
(30,15)*+{B}="b";
(45.5,-23)*+{\non{Im}(pi)}="im2";
(7.5,-7.5)*+{\non{{}_{(c)}}};
(23,-15)*+{\non{{}_{(d)}}};
{\ar@{_{(}->}_{i} "b"; "c123"};
{\ar@{->} "c3"; "*23"};
{\ar@{^{(}->} "*23"; "c2"};
{\ar@{->} "c2"; "*12"};
{\ar@{^{(}->} "*12"; "c1"};
{\ar@{^{(}->} "c3"; "c23"};
{\ar@{^{(}->} "c23"; "c123"};
{\ar@{->} "c23"; "c2"};
{\ar@{->}^{p} "c123"; "c1"};
{\ar@{->>}@/^{1.25pc}/ "b"; "im2"};
{\ar@{^{(}->} "im2"; "c1"}
\endxy}}
\ \ \ \ \ \ \ \ \ \mathrm{and}\ \ \ \ \ \ \ \ \ \ \ \ \ 
\vcenter{\vbox{
\xy
(0,0)*+{C_3}="c3";
(30,0)*+{C_{123}}="c123";
(0,-15)*+{\ast}="*23";
(15,-15)*+{C_2}="c2";
(30,-15)*+{C_{12}}="c12";
(15,-30)*+{\ast}="*12";
(30,-30)*+{C_1}="c1";
(30,15)*+{B}="b";
(45.5,-7.50)*+{\non{Im}(p_{2}i)}="im1";
(45.5,-23)*+{\non{Im}(p_{1}k)}="im2";
(23,-23)*+{\non{{}_{(a)}}};
(15,-7.5)*+{\non{{}_{(b)}}};
{\ar@{^{(}->}_{i} "b"; "c123"};
{\ar@{->} "c3"; "*23"};
{\ar@{^{(}->} "*23"; "c2"};
{\ar@{->} "c2"; "*12"};
{\ar@{^{(}->} "*12"; "c1"};
{\ar@{^{(}->} "c3"; "c123"};
{\ar@{^{(}->} "c2"; "c12"};
{\ar@{->}_{p_2} "c123"; "c12"};
{\ar@{->}^{p_1} "c12"; "c1"};
{\ar@{->>}@/^{1.4pc}/ "b"; "im1"};
{\ar@{->>} "im1"; "im2"};
{\ar@{^{(}->}^{k\ \ \ \ } "im1"; "c12"};
{\ar@{^{(}->} "im2"; "c1"}
\endxy}}
$$
When we compute the double-composition associated to the diagram at left,
we get the evaluation of
$\eta_{1}:F_{C_1}\Longrightarrow\!\!\!\!\Rightarrow E_1$
at the object $\mathrm{Im}(pi)\hookrightarrow C_{1}$
in $\bold{Mon}_{C_{1}}\Cc$.
When we compute the double-composition associated to the diagram at right,
we see that it returns the evaluation of
$\eta_{1}:F_{C_1}\Longrightarrow\!\!\!\!\Rightarrow E_1$
at the object $\mathrm{Im}(p_{1}k)\hookrightarrow C_{1}$
in $\bold{Mon}_{C_{1}}\Cc$. 
By Lemma \ref{gi6}, the objects
$\mathrm{Im}(pi)\hookrightarrow C_{1}$ and $\mathrm{Im}(p_{1}k)\hookrightarrow C_{1}$
in $\bold{Mon}_{C_{1}}\Cc$ are isomorphic.

\vskip .3cm

Finally, to verify axiom {\bf OP-3},
fix a polyhedral diagram in $\Cc$ as in axiom {\bf OP-3}
of Definition \ref{absop},
and fix an object
$B\hookrightarrow A$
in
$\bold{Mon}_{A}\mathcal{C}$.
Consider a triad of epimorphisms
$$\eta:F_{D}\Longrightarrow\!\!\!\!\Rightarrow E{\non{\ \ \ \ \ \ over\ \ \ \ \ \ }}\bold{Mon}_{D}\mathcal{C}$$
$$\eta_{1}:F_{C_1}\Longrightarrow\!\!\!\!\Rightarrow E_1{\non{\ \ \ \ \ \ over\ \ \ \ \ \ }}\bold{Mon}_{C_1}\mathcal{C}$$
$$\eta_{2}:F_{C_2}\Longrightarrow\!\!\!\!\Rightarrow E_2{\non{\ \ \ \ \ \ over\ \ \ \ \ \ }}\bold{Mon}_{C_2}\mathcal{C},$$
representing an element in
$\FM{F}(D)\times\FM{F}(C_1)\times\FM{F}(C_2)$.

We have to consider two cases:

\vskip .2cm

$\non{\bold{Case-1:}}$ If $B\hookrightarrow A$
factors through $C_{1}\hookrightarrow A$,
then consider the pair of commutative diagrams
$$
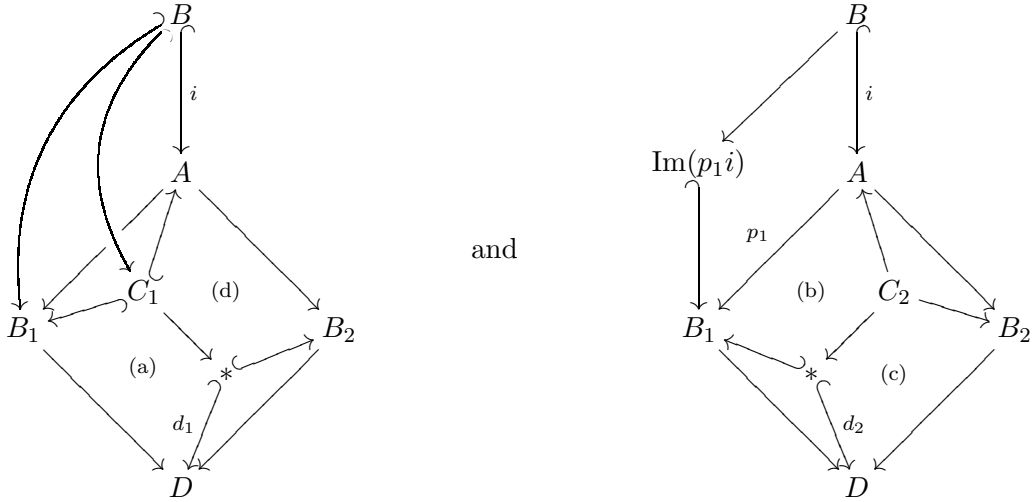
\vcenter{\vbox{
\xy
(21,63)*+{B}="b";
(21,42)*+{A}="a";
(0,21)*+{B_{1}}="b1";
(42,21)*+{B_{2}}="b2";
(21,0)*+{D}="d";
(16,26)*+{C_{1}}="c1";
(27,15)*+{\ast}="*";
(16,16)*+{{}_{\non{(a)}}}="(a)";
(27,26)*+{{}_{\non{(d)}}}="(d)";
{\ar@{->}|(.425){\hole} "a";"b1"};
{\ar@{_{(}->}@/_{2pc}/ "b";"b1"};
{\ar@{->} "a";"b2"};
{\ar@{->} "b1";"d"};
{\ar@{->} "b2";"d"};
{\ar@{_{(}->} "c1";"a"};
{\ar@{^{(}->} "c1";"b1"};
{\ar@{->} "c1";"*"};
{\ar@{^{(}->} "*";"b2"};
{\ar@{_{(}->}_{d_1} "*";"d"};
{\ar@{^{(}->}^{i} "b"; "a"};
{\ar@{^{(}->}@/_{2pc}/ "b"; "c1"}
\endxy}}
\ \ \ \ \ \ \ \ \ \ \ \mathrm{and}\ \ \ \ \ \ \ \ \ \ \ \ \ 
\vcenter{\vbox{
\xy
(21,63)*+{B}="b";
(0,43)*+{\mathrm{Im}(p_{1}i)}="Im";
(21,42)*+{A}="a";
(0,21)*+{B_{1}}="b1";
(42,21)*+{B_{2}}="b2";
(21,0)*+{D}="d";
(26,26)*+{C_{2}}="c2";
(15,15)*+{\ast}="*";
(26,15)*+{{}_{\non{(c)}}}="(c)";
(15,26)*+{{}_{\non{(b)}}}="(b)";
{\ar@{->}_{p_{1}} "a";"b1"};
{\ar@{->} "b";"Im"};
{\ar@{_{(}->} "Im";"b1"};
{\ar@{->} "a";"b2"};
{\ar@{->} "b1";"d"};
{\ar@{->} "b2";"d"};
{\ar@{->} "c2";"a"};
{\ar@{->} "c2";"b2"};
{\ar@{->} "c2";"*"};
{\ar@{_{(}->} "*";"b1"};
{\ar@{^{(}->}^{d_2} "*";"d"};
{\ar@{^{(}->}^{i} "b"; "a"};
\endxy}}
$$
Observe that since $p_{1}i$ is none other than the composite
$B\hookrightarrow C_{1}\hookrightarrow B_{1}$,
it is a monomorphism.
Hence the object
$\mathrm{Im}(p_{1}i)\hookrightarrow B_{1}$
in $\bold{Mon}_{B_1}\Cc$ is isomorphic to
$B\hookrightarrow B_{1}$.
It follows from this that the double-composition gotten by first applying the composition associated to
(a) and then applying the composition associated to (b)
returns the evaluation of
$\eta_{1}:F_{C_1}\Longrightarrow\!\!\!\!\Rightarrow E_1$
at the object $B\hookrightarrow C_{1}$
in $\bold{Mon}_{C_1}\Cc$.
On the other hand, the double-composition gotten by first applying the composition associated to
(c) and then applying the composition associated to (d)
also returns the evaluation of
$\eta_{1}:F_{C_1}\Longrightarrow\!\!\!\!\Rightarrow E_1$
on the object $B\hookrightarrow C_{1}$
in $\bold{Mon}_{C_1}\Cc$.

By symmetry, this also accounts for the case where $B\hookrightarrow A$
factors through $C_2$.

\vskip .2cm

{\bf Case-2:} If $B\hookrightarrow A$
does not factor through $C_{1}\hookrightarrow A$,
then consider the commutative diagram
$$
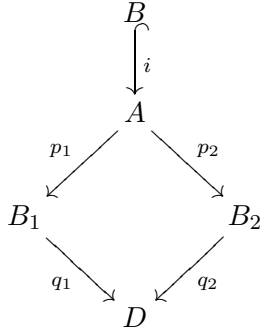
\xymatrix{
&B\ar@{^{(}->}^{i}[d]
\\
&A\ar_{p_1}[ld]\ar^{p_2}[rd]
\\
B_{1}\ar_{q_1}[rd]
&&
B_{2}\ar^{q_2}[ld]
\\
&
D
}
$$
Since $p_1$ and $p_2$ are acute quotients,
realized by the acute bicartesian squares (b) and (d), respectively,
Lemma \ref{gi6} implies that this diagram determines a unique image 
$\mathrm{Im}\hookrightarrow D$
for $B$ in $D$,
and both the (a)-(b) double-composition and the (c)-(d) double composition
return the evaluation of
$\eta_{D}:F_{C_D}\Longrightarrow\!\!\!\!\Rightarrow E$
on this object $\mathrm{Im}\hookrightarrow D$ in
$\bold{Mon}_{D}\Cc$.
\end{proof}

\vskip 2cm

\section{Fulton-MacPherson operads represented by quiver Grassmannians.}\label{q grass}

\vskip .5cm

We begin this section with another example of a right-acute functor and its associated Fulton-MacPherson operad. We will give more examples, variants on the present example, at the end of this section.

\subsection{Conormal functors \& their operads.}
Fix a ground field $\Bbbk$, and let
$\FVect{\Bbbk}$
denote the category of finite dimensional
$\Bbbk$-vector spaces.
Recall that $\fnS$ denotes the category of finite, nonempty sets,
and that $\fnS$ is a good input category.
It turns out that we can construct all right-acute functors
$F:\fnS\lra\FVect{\Bbbk}$,
in fact all acute functors of this form, 
in the following manner:

\vskip .3cm

\begin{definition}
{\bf (Conormal functor)}
For each finite dimensional vector space $V$,
the {\em conormal functor}
associated to $V$
is the functor
$$\con{V}:\fnS\lra\FVect{\Bbbk}$$
that, at each nonempty, finite set
$I$, returns the kernel
$\con{V}(I)$
of the addition map
$V^{I}\lra V$:
$$
\xy
(0,0)*+{0}="1";
(15,0)*+{\con{V}(I)}="2";
(30,0)*+{V^I}="3";
(45,0)*+{V}="4";
(60,0)*+{0}="5";
(30,-10)*+{(v_{i})_{I}}="6";
(45,-10)*+{\Sigma_{I}{v_i}}="7";
(30,-5)*+{\raisebox{\depth}{\rotatebox[origin=c]{90}{$\in$}}\ };
(45,-5)*+{\raisebox{\depth}{\rotatebox[origin=c]{90}{$\in$}}};
{\ar@{->} "1"; "2"};
{\ar@{->} "2"; "3"};
{\ar@{->} "3"; "4"};
{\ar@{->} "4"; "5"};
{\ar@{|->} "6"; "7"};
\endxy
$$
\end{definition}

\vskip .3cm

\begin{lemma}\label{conorm}
{\bf (Conormal functors -vs.- acute functors.)}

{\bf (1)} The conormal functor $\con{V}:\fnS\lra\FVect{\Bbbk}$ is acute.

{\bf (2)} Conversely, if $F:\fnS\lra\FVect{\Bbbk}$ is a right-acute functor, then
$F\cong\con{V}$,
where $V$ is the value of $F$ on any fixed, two-element set:
$V=F(\{\ast_1,\ast_2\})$.
\end{lemma}

In particular, if $F:\fnS\lra\FVect{\Bbbk}$ is right-acute, then it is acute.

\begin{proof}
  The verification of {\bf (1)} is straightforward.
  
  To verify {\bf (2)}, fix a two element set
  $\{\bullet,\ast\}$ in $\bold{S}$
  once and for all.
  Define
  $V=F(\{\bullet,\ast\})$,
  and define a new functor
  $$\Omega^{1}:\fnS\lra\FVect{\Bbbk}$$
  according to
  $\Omega^{1}(I)=F(I\sqcup\ast)$.
  
  Note that $V=\Omega^{1}(\bullet)$.
  
  For each element $i\in I$,
  let $\varphi_{i}:\{\bullet,\ast\}\hookrightarrow I\sqcup\ast$
  be the inclusion that fixes $\ast$ and takes $\bullet$ to the element $i$.
  Let $\pi:I\sqcup\ast\lra\{\bullet,\ast\}$
  be the projection that fixes $\ast$ and takes the entire set $I$ to $\bullet$.
  Since each $\varphi_{i}$
  is a section of $\pi$,
  each
  $\varphi_{i}$
  induces an inclusion $F(\varphi_{i}):V\hookrightarrow\Omega^{1}(I)$.
  We claim that $\Omega^{1}(I)$
  is the direct sum of the images of these inclusions.
  Indeed, choose a total ordering $I=\{i_{1},i_{2},\dots,i_{n}\}$,
  and consider the sequence
  $$
  I\sqcup\ast\xrightarrow{\pi_1}\{i_{2},i_{3},\dots,i_{n},\ast\}\xrightarrow{\pi_2}\{i_{3},\dots,i_{n},\ast\}\xrightarrow{\pi_3}\cdots\xrightarrow{\pi_{n-1}}\{i_{n},\ast\}
  $$
  wherein $\pi_{m}$ fixes the subset
  $\{i_{m+1},\dots,i_{n},\ast\}$,
  and takes $i_m$ to $\ast$.
  Since $F$ is right-acute by assumption,
  the image of this last sequence under $F$
  becomes a tower whose $m^\mathrm{th}$ step is the extension
  $$
  0\lra V\lra \Omega^{1}(\{i_{m},i_{m-1},\dots,i_{n}\})\xrightarrow{F(\pi_m)}\Omega^{1}(\{i_{m-1},\dots,i_{n}\})\lra 0
  $$
  Thus $\Omega^{1}(I)\cong V^I$.
  
  With this isomorphism in hand, right-acuteness of $F$
  gives us an exact sequence
  $$
  F(I)\lra V^{I}\xrightarrow{F(\pi)}V\lra 0
  $$
  Using any projection $I\sqcup\ast\lra\!\!\!\!\rightarrow I$ that preserves $I$,
  we see that $F(I)\longrightarrow V^I$ must be injective.
  Furthermore, since each inclusion $\varphi_i$ is a section of $\pi$,
  the morphism $F(\pi):V^{I}\longrightarrow V$
  must be the addition morphism.
\end{proof}

\vskip .3cm

\begin{example}\label{config op}
{\bf (The operad of configurations in {\em V}.)}
It is easy to see that each conormal functor
$\con{V}:\fnS\lra\FVect{\Bbbk}$
has locally defined, locally small relations over
$\fnS$.
Thus by Lemma \ref{conorm} and Theorem \ref{FMop},
the Fulton-MacPherson construction,
applied to $\con{V}$,
produces a classical, set-valued operad
$$
\FM{\con{V}}:\isop{\fnS}\longrightarrow\Sets,
$$
which we call the {\em conormal operad}
associated to $V$.

In fact, we can view the assignment
$V\mapsto\FM{\con{V}}$
as a functor
$\FVect{\Bbbk}\lra\bold{Op}(\fnS,\Sets)$.
In particular, each operad
$\FM{\con{V}}$
comes with a right
$\mathrm{PGL}(V)$-action that is compatible with the operadic structure on
$\FM{\con{V}}$,
and which commutes with the right
$\Aut_{\fnS}(I)$-action
on each set $\FM{\con{V}}(I)$.

\vskip .4cm

Our central result in the present section is that the conormal operad
$\FM{\con{V}}$
is represented by an operad with values in projective
$\Bbbk$-varieties.
Many operads like $\FM{C_V}$ are similarly
represented by projective $\Bbbk$-varieties or projective
$\Bbbk$-schemes. We will give examples below.

Let $\bold{Sch}_{\Bbbk}$
denote the category of finite-type $\Bbbk$-schemes,
and let $\bold{Var}_{\Bbbk}$ denote its full subcategory of projective
$\Bbbk$-varieties.
Then our prototypical Theorem is the following:

\end{example}

\vskip .5cm

\begin{theorem}\label{representability}
{\bf (Representability of the conormal operad.)}
For each finite dimensional $\Bbbk$-vector space $V$,
there exists an operad
$\non{FM}_{\con{V}}:\isop{\fnS}\longrightarrow\bold{Var}_{\Bbbk}$
with values in projective $\Bbbk$-varieties,
which represents the conormal operad $\FM{\con{V}}$,
in the sense that there exist natural isomorphisms
$$
\FM{\con{V}}(I)\ \ \cong\ \ \Hom_{\Sch{\Bbbk}}(\Spec{\Bbbk},\mathrm{FM}_{\con{V}\!}(I))
$$
compatible with the two operadic structures.
\end{theorem}

\vskip .3cm

\begin{remark}
We will construct each of these projective $\Bbbk$-varieties
$\mathrm{FM}_{\con{V}}(I)$
as a {\em quiver Grassmannian}.
Let us spend a moment reviewing the theory of quiver Grassmannians.
\end{remark}

\vskip .3cm

\subsection{Reminder on quivers \& quiver Grassmannians}\ 

{\bf Quivers and their representations.}
Recall that a {\em quiver} $\Qc$ is nothing but an oriented graph.
More explicitly, a quiver $\Qc$ consists of a set $\Qc_0$
of {\em vertices},
a set $\Qc_1$
of {\em arrows},
and {\em source} and {\em target} functions
$$
\xy
(0,0)*+{\Qc_1}="1";
(15,0)*+{\Qc_0}="0";
{\ar@<1 ex>^{\mathrm{src}} "1"; "0"};
{\ar@<-.5 ex>_{\mathrm{tar}} "1"; "0"}
\endxy
$$
We will communicate the fact that an arrow
$a\in\Qc_1$
has source $u=\mathrm{src}(a)$
and target $v=\mathrm{tar}(a)$
in $\Qc_0$
by writing
$u\ \bullet\xrightarrow{\ \ a\ \ }\bullet\ v$.

Let $\Qc$ be any quiver,
and let $\Dc$ be some fixed category.
Then a {\em representation} $\qr{M}$
of $\Qc$
in $\Dc$
consists of an assignment of an object
$\qr{M}(v)$ of $\Dc$
to each vertex $v$ in $\Qc$,
and an assignment of a morphism
$\qr{M}(a):\qr{M}(u)\lra\qr{M}(v)$ in $\Dc$
to each arrow
$u\ \bullet\xrightarrow{\ \ a\ \ }\bullet\ v$
in $\Qc$.

A {\em morphism}
$\psi:\qr{M}\lra\qr{M}'$
between two representations of $\Qc$ in $\Dc$
is any family of morphisms
$\psi_{v}:\qr{M}(v)\lra\qr{M}'(v)$ in $\Dc$,
indexed by the vertices of $\Qc$,
that commute in the obvious way with the morphisms in
$\qr{M}$ and $\qr{M}'$.
We let
$\bold{Rep}_{\Dc}(\Qc)$
denote the category of representations of
$\Qc$ in $\Dc$.

Fix a quiver
$\Qc$.
To ease notation, let
$\bold{Rep}_{\Bbbk}(\Qc)$
denote the category of representations of $\Qc$
in $\FVect{\Bbbk}$, that is, in finite dimensional
$\Bbbk$-vector spaces.
The category $\bold{Rep}_{\Bbbk}(\Qc)$
is abelian, with kernels and cokernels taken vertexwise.
If our underlying quiver $\Qc$
is {\em finite}, meaning that $\Qc_0$ and $\Qc_1$ are both finite,
and is {\em acyclic},
meaning that every oriented path in $\Qc$ has a distinct source and target,
then the Grothendieck group
$\mathrm{K}_{0}(\Qc)$ of
$\bold{Rep}_{\Bbbk}(\Qc)$
is isomorphic to the free
$\ZZ$-module $\ZZ\Qc_0$ generated by the set
$\Qc_0$ of vertices in our quiver $\Qc$.
By a {\em dimension vector} $\qr{d}$ on $\Qc$,
we mean any element in $\ZZ\Qc_0$.
Given a dimension vector $\qr{d}$,
we let $\qr{d}(v)\in\ZZ$
denote its coefficient at the vertex $v$ in $\Qc$.
The isomorphism $\ZZ\Qc_{0}\cong\mathrm{K}_{0}(\Qc)$
identifies dimension vectors with classes in
$\mathrm{K}_{0}(\Qc)$, and vice versa.
Thus given any representation
$\qr{M}$ of $\Qc$
in $\FVect{\Bbbk}$,
we let $[\qr{M}]\in\ZZ\Qc_{0}$
denote the associated dimension vector of $\qr{M}$.
Its coefficient at any given vertex $v\in\Qc_0$
is none other than $\mathrm{dim}_{\Bbbk\ }\qr{M}(v)$.

Fix a ground field $\Bbbk$,
a quiver $\Qc$,
and a representation $\qr{M}$ of $\Qc$
in $\FVect{\Bbbk}$.
Then a {\em subrepresentation}
$\qr{N}$ of $\qr{M}$
consists of a $\Bbbk$-linear subspace
$\qr{N}(v)\ \sub\ \qr{M}(v)$
at each vertex $v$ of $\Qc$,
such that at every morphism
$\qr{M}(a):\qr{M}(u)\lra\qr{M}(v)$
in our representation,
we have
$\qr{M}(a)\big(\qr{N}(u)\big)\sub\qr{M}(v)$.
Equivalently, a subrepresentation is the image of a monomorphism
$\qr{N}\hookrightarrow\qr{M}$
in $\bold{Rep}_{\Bbbk}(\Qc)$.
In particular,
every subrepresentation of $\qr{M}$
is a representation of $\Qc$.
We say that $\qr{N}$
is a
$\qr{d}$-{\em dimensional subrepresentation}
of $\qr{M}$
is $\qr{d}(v)=\mathrm{dim}_{\Bbbk\ }\qr{N}(v)$
at each vertex $v$ in $\Qc$,
that is, if $[\qr{N}]=\qr{d}$.

\vskip .3cm

{\bf Quiver Grassmannians.}
Recall that $\bold{Sch}_{\Bbbk}$ denotes the category of finite-type
$\Bbbk$-schemes.
Over each scheme $T$, let $\bold{Coh}_T$
denote the category of coherent $\mathscr{O}_T$-modules.

Fix a finite quiver $\Qc$,
a representation $\qr{M}$ of $\Qc$
in $\FVect{\Bbbk}$,
and a dimension vector $\qr{d}$ on $\Qc$.
Then there exists a functor
$$
\mathscr{G}\!r_{\Qc}^{\mathrm{nr}}(\qr{d},\qr{M}):\bold{Sch}_{\Bbbk}^\mathrm{op}\lra\Sets
$$
that, at each $\Bbbk$-scheme $\pi:T\lra\Spec{\Bbbk}$,
returns the set
$$
\mathscr{G}\!r_{\Qc}^{\mathrm{nr}}(\qr{d},\qr{M})(T)
\ \ =\ \ 
\left\{
	\begin{array}{l}
		\mathrm{quotients\ }\pi^{*}\qr{M}\lra\!\!\!\!\rightarrow\qr{E}\mathrm{\ in\ }\bold{Rep}_{\bold{Coh}/T}(\Qc)\mathrm{,\ such}
		\\
		\mathrm{that,\ at\ each\ vertex\ }v\in\Qc_{0},\mathrm{\ the\ }\mathscr{O}_{T}\textendash\mathrm{module}
		\\
		\qr{E}(v)\mathrm{\ is\ locally\ free\ of\ rank\ }\mathrm{dim}_{\Bbbk\ }\qr{M}(v)-\qr{d}(v)
	\end{array}
\right\}\Big/\mathrm{isomorphism}.
$$
It follows immediately from the definition that this functor
$\mathscr{G}\!r_{\Qc}^{\mathrm{nr}}(\qr{d},\qr{M})$
is represented by a closed subscheme
$$
\mathrm{Gr}_{\Qc}^{\mathrm{nr}}(\qr{d},\qr{M})
$$
in the product $\prod_{v\in\Qc_{0}}\mathrm{Gr}(\qr{d}(v),\qr{M}(v))$
of classical Grassmannians.
This subscheme
$\mathrm{Gr}_{\Qc}^{\mathrm{nr}}(\qr{d},\qr{M})$
need not be reduced in general.
We will refer to it as the
{\em non-reduced quiver Grassmannian}.

The {\em quiver Grassmannian}
$$
\mathrm{Gr}_{\Qc}(\qr{d},\qr{M})
$$
is, by definition, the reduced $\Bbbk$-scheme underlying
$\mathrm{Gr}_{\Qc}^{\mathrm{nr}}(\qr{d},\qr{M})$.
Thus $\mathrm{Gr}_{\Qc}(\qr{d},\qr{M})$
is, in particular, a projective $\Bbbk$-variety that is not necessarily connected or irreducible.
Most importantly for our purposes,
note that the $\Bbbk$-valued points of
$\mathrm{Gr}_{\Qc}(\qr{d},\qr{M})$
are in bijection with $\qr{d}$-dimensional
subrepresentations
$\qr{N}$ of $\qr{M}$.
For details, see \cite{C-B} and \cite{Sch}.

To give an example:
If $\Qc$ is the quiver
$\Qc=\bullet$
consisting of a single vertex with no arrows,
then a representation $\qr{M}$ of $\Qc$
in $\FVect{\Bbbk}$
amounts to a single $\Bbbk$-vector space
$M=\qr{M}(\bullet)$,
and a dimension vector $\qr{d}$ on $\Qc$
amounts to a single integer $d=\qr{d}(\bullet)$.
The resulting quiver Grassmannian is nothing but the classical Grassmannian
$\mathrm{Gr}_{\Qc}(\qr{d},\qr{M})=\mathrm{Gr}(d,M)$,
parametrizing
$d$-dimensional subspaces of $M$.

\vskip .5cm

\subsection{Construction of $\bold{FM}_{\bold{\con{V}}}$}
Consider the following special case of a quiver,
a representation of the quiver, and its quiver Grassmannians:

Fix a ground field $\Bbbk$ and a finite dimensional vector space $V$.

For each nonempty, finite set $I$,
let the quiver $\Qc(I)$ {\em associated to} $I$
be the quiver whose vertices are the nonempty subsets $J\sub I$,
and whose arrows are proper inclusions $J'\subset J$ of nonempty subsets of $I$.
We've already seen that our vector space $V$
determines a conormal functor
$$
\con{V}:\fnS\lra\FVect{\Bbbk}
$$
This conormal functor $\con{V}$ induces, in turn,
a representation $\conqr{V}$
of the quiver $\Qc(I)$,
namely the representation $\conqr{V}$ taking each nonempty subset
$J\sub I$ to the $\Bbbk$-vector space
$\conqr{V}(J):=\con{V}(J)$,
and taking each proper inclusion $\jmath:J'\hookrightarrow J$
of nonempty subsets of $I$
to the $\Bbbk$-linear map
$\conqr{V}(\jmath):=\con{V}(\jmath):\con{V}(J')\hookrightarrow\con{V}(J)$
of values of the conormal functor.

We call $\conqr{V}$
the {\em conormal representation}
of the quiver $\Qc(I)$.

With this representation in hand,
each dimension vector $\qr{d}$ on $\Qc(I)$
determines a quiver Grassmannian:
$$
\mathrm{Gr}_{\Qc(I)}(\qr{d},\conqr{V})
$$
Let us consider all possible choices of the dimension vector
$\qr{d}$ simultaneously,
by taking the disjoint union of all the quiver Grassmannians
$\mathrm{Gr}_{\Qc(I)}(\qr{d},\conqr{V})$.
We define $\mathrm{FM}_{\con{V}}(I)$
to be this disjoint union:
$$
\mathrm{FM}_{\con{V}\!}(I)\ \ :=\bigsqcup_{\qr{d}\in\ZZ\Qc_{{}_{0\!}}(I)}\!\!\!\!\mathrm{Gr}_{\Qc(I)}(\qr{d},\conqr{V})
$$
Notice that this disjoint union has only finitely many nonempty terms,
since the quiver Grassmannian
$\mathrm{Gr}_{\Qc(I)}(\qr{d},\conqr{V})$
can only be nonempty if
$0\leq\qr{d}(J)\leq\mathrm{dim}_{\Bbbk\ }\conqr{V}(J)$
for each nonempty subset
$J\sub I$.

\vskip .5cm

\begin{proof}
{\bf (Proof of Theorem \ref{representability})}
Fix a nonempty, finite set $I$.
Recall that $\Mon{I}{\fnS}$
denotes its poset of subobjects in the category
$\bold{S}$ of nonempty finite sets.

Let $\qr{N}\subset\conqr{V}$
be an arbitrary subrepresentation.

Each object
$\iota:J\hookrightarrow I$ in
$\Mon{I}{\fnS}$
comes with a canonical isomorphism
$\con{V}(J)\xrightarrow{\ \sim\ }\con{V}(\iota(J))$.
Thus the subrepresentation
$\qr{N}\subset\conqr{V}$
determines the quotient
$E(J\!\hookrightarrow\!I)$ appearing in the short exact sequence
$$
0\lra\qr{N}(\iota(J))\lra\con{V}(J)\lra E(J\!\hookrightarrow\!I)\lra0
$$
It follows immediately from the defining property of any subrepresentation of
$\conqr{V}$
that the epimorphisms
$\con{V}(J)\lra\!\!\!\!\rightarrow E(J\!\hookrightarrow\!I)$
obtained in this way constitute an object in
$\bold{Epi}^{\con{V}}[\Mon{I}{\fnS},\FVect{\Bbbk}]$.

Conversely, when we restrict $\Mon{I}{\fnS}$
to its subposet consisting of {\em actual} nonempty subsets $J\sub I$,
the kernels of these epimorphisms
$\conqr{V}(J)\lra\!\!\!\!\rightarrow E(J\!\hookrightarrow\!I)$
recover our original subrepresentation
$\qr{N}\subset\conqr{V}$.
Thus we have bijections
$$
\FM{\con{V}\!}(I)\ \ \cong\ \ \Hom_{\Sch{\Bbbk}}(\Spec{\Bbbk},\mathrm{FM}_{\con{V}\!}(I))
$$

To define an operadic structure on the family of projective varieties
$\mathrm{FM}_{\con{V}\!}(I)$,
consider the sheaf-valued functor
$\mathrm{FM}_{\con{V}\!}^{\mathrm{nr}}:\isop{\fnS}\lra\bold{Sh}(\bold{Sch}_{\Bbbk})$
that takes each nonempty finite set $I$
to the sheaf over $\bold{Sch}_{\Bbbk}$
represented by the disjoint union of {\em non-reduced} quiver Grassmannians:
$$
\mathrm{FM}_{\con{V}\!}^{\mathrm{nr}}(I)
\ \ :=
\bigsqcup_{\qr{d}\in\ZZ\Qc(I)}\!\!
\mathrm{Gr}_{\Qc(I)}^{\mathrm{nr}}(\qr{d},\conqr{V})
$$
It follows from the definition of the non-reduced quiver Grassmannian
$\mathrm{Gr}_{\Qc(I)}^{\mathrm{nr}}(\qr{d},\conqr{V})$
that, over each finite-type $\Bbbk$-scheme
$\pi:T\lra\Spec{\Bbbk}$,
the set-valued functor
$\mathrm{FM}_{\con{V}\!}^{\mathrm{nr}}(-)(T):\isop{\fnS}\lra\Sets$
is the result of a Fulton-MacPherson construction for the right-acute functor
$\pi^{*}\con{V}:\fnS\lra\bold{Coh}_{T}$.
Thus Theorem \ref{FMop}, combined with Yoneda's Lemma,
gives an operadic structure on the family of non-reduced
$\Bbbk$-schemes
$\mathrm{FM}_{\con{V}\!}^{\mathrm{nr}}(I)$.
This operadic structure restricts to an operadic structure on the reduced
$\Bbbk$-schemes that underly the
$\mathrm{FM}_{\con{V}\!}^{\mathrm{nr}}(I)$,
that is, to an operadic structure on the varieties
$\mathrm{FM}_{\con{V}\!}(I)$,
and it clearly coincides with the operadic structure induced by the Fulton-MacPherson construction for $\con{V}$.
\end{proof}

\vskip .5cm

\subsection{Relation to Chen-Gibney-Krashen operads and operads of stable curves}
Each Chen-Gibney-Krashen operad $\mathrm{CGK}_{V}$ is a suboperad of the operad $\mathrm{FM}_{\con{V}}$.

Specifically, for each nonempty finite set $I$,
let $\qr{scr}$ denote the dimension vector on $\Qc(I)$ that, at each nonempty subset $J\sub I$, returns
$$
\qr{scr}(J)=
\left\{
	\begin{array}{ll}
		\mathrm{dim}_{\Bbbk}(\conqr{V}(J))-1  & \mbox{if} \ \ |J|>1 \\
		\ \ \ \ \ \ \ \ \ \ 0 & \mbox{if} \ \ |J|=1
	\end{array}
\right.
$$
Then the component $\mathrm{Gr}_{\Qc(I)}(\qr{scr},\conqr{V})$
of $\mathrm{FM}_{\con{V}\!}(I)$
is canonically isomorphic the Chen-Gibney-Krashen space $\mathrm{CGK}_{V}(|I|)$, since the functors of points of these two varieties coincide.
Moreover, the operadic structure on
$\non{FM}_{\con{V}}:\isop{\fnS}\longrightarrow\bold{Var}_{\Bbbk}$
restricts to an operadic structure on the varieties
$\mathrm{Gr}_{\Qc(I)}^{\mathrm{nr}}(\qr{scr},\conqr{V})$,
and this structure coincides with the Chen-Gibney-Krashen operadic structure.

When $V$ is 1-dimensional,
there is an identification $\mathrm{CGK}_{V}(|I|)=\overline{M}_{0,|I|+1}$,
where $\overline{M}_{0,|I|+1}$
denotes the moduli space of stable, $(|I|+1)$-marked rational curves, and this identification constitutes an isomorphism of operads
(see \cite{CGK} for details).
Thus $\mathrm{FM}_{\con{V}}$
contains the operad of stable, marked rational curves as a suboperad
when $\mathrm{dim}_{\Bbbk}V=1$.

Further properties of Chen-Gibney-Krashen operads extend to the operads
$\mathrm{FM}_{\con{V}}$ as well,
but we leave these details for a subsequent paper.

\vskip .7cm

\subsection{Further examples}
Theorem \ref{representability} and its proof
set the pattern for proofs that many examples of Fulton-MacPherson operads are representable.
Rather than attempt to prove some general statement,
we provide several further examples,
illustrating how one modifies the proof of Theorem \ref{representability}
in each case:

\vskip .5cm

\begin{example}
{\bf (A representable, abstract operad with inputs of "type-B.")}

Let $\bold{B}$ denote the category whose objects are
$\ZZ$-lattices of the form $\ZZ^I$,
for any finite (possibly empty) set $I$,
and whose morphisms
$\ZZ^{I}\lra\ZZ^{J}$
are given by {\em extended signed permutation matrices} in
$\mathrm{Mat}_{|I|\times|J|}(\ZZ)$,
i.e., those $|I|\times|J|$-matrices $M$
whose every entry is either
$0$, $1$, or $-1$,
such that each row and each column in $M$
contains at most one nonzero entry.

This category $\bold{B}$
can be thought of as a category of type-B root systems.
For instance, if $I$
is a finite set containing at least two elements, then
$\mathrm{Aut}_{\bold{B}}(\ZZ^I)\cong\mathcal{W}(B_{|I|})$,
the Weyl group for the root system $\mathrm{B}_{|I|}$.
Notice also that homsets in
$\bold{B}$
are finite.
In fact,
$\bold{B}$
behaves like the category
$\fnS$
in many respects.
Most important to us:

\begin{lemma}
The category $\bold{B}$
is a good input category.
\end{lemma}

The proof is straight-forward. We leave its details to the reader.

\vskip .3cm

Fix a finite dimensional
$\Bbbk$-vector space $V$,
and consider the functor $V\!\tensor_{\ZZ}-:\bold{B}\lra\FVect{\Bbbk}$
taking $\ZZ^{I}\mapsto V^{I}$.
This functor is acute,
with locally defined, locally small relations over
$\bold{B}$.
Hence by Theorem \ref{FMop}, the Fulton-MacPherson construction for $V\!\tensor_{\ZZ}-$
produces an abstract, set-valued operad
$$
\FM{V\!\tensor_{\ZZ}-}:\isop{\bold{B}}\lra\Sets
$$
For each finite (possibly empty) set $I$,
let $\Bc(I)$ be the quiver whose vertices are all subsets $J\sub I$,
with arrows corresponding to all proper inclusions $J\subset J'$ in $I$.
Then the functor $V\!\tensor_{\ZZ}-$ induces a representation
$\qr{V}^{I}$ of $\Bc(I)$.
By essentially the same argument in our proof of Theorem \ref{representability},
with $\qr{V}^{I}$ replacing $\conqr{V}$,
one shows that each set $\mathcal{FM}_{V\!\tensor_{\ZZ}-}(\ZZ^{I})$
is in bijection with the set of
$\Bbbk$-points of the projective $\Bbbk$-variety
$$
\mathrm{FM}_{V\!\tensor_{\ZZ}-}(\ZZ^I)
\ \ \ =
\bigsqcup_{\qr{d}\in\ZZ\Bc_{{}_{0\!}}(I)}
\!\!\!\!\mathrm{Gr}_{\Bc(I)}(\qr{d},\qr{V}^{I})
$$
The resulting functor
$\mathrm{FM}_{V\!\tensor_{\ZZ}-}:\isop{\bold{B}}\lra\bold{Var}_\Bbbk$
admits an abstract operadic structure that recovers the abstract operadic structure of $\mathcal{FM}_{V\!\tensor_{\ZZ}-}\ $
at the level of $\Bbbk$-points.

This example has many generalizations.
For example, we can interpret $\bold{B}$
as {\em the category of} $\mathbb{F}_{1^2}$-{\em modules}
in Kapranov and Smirnov's theory of linear algebra over the (nonexistent) finite extensions $\mathbb{F}_{1^{n}}$
of the (nonexistent) field $\mathbb{F}_{1}$
with one element.
We can then replace the category $\bold{B}$
with {\em the category of
$\mathbb{F}_{1^{n}}$-modules},
for any integer $n\geq2$.
See \cite{KS} for details.
All the above results then hold upon replacing our ground field
$\Bbbk$ with the field
$\Bbbk(\mu_{n})$,
the extension of $\Bbbk$ by the $n^\mathrm{th}$-roots of unity.

\end{example}

\begin{example}
{\bf (Representable operads with inputs in simplicial sets.)}
Fix a smooth, projective $\Bbbk$-variety $X$.
Recall that
$\textsf{F}s\bold{S}$
is the category of finite-dimensional simplicial sets
$I_{\bullet}$
for which every $I_n$ is finite and nonempty.
Let $s\bold{Coh}_{X}$ denote the category of simplicial coherent
$\mathscr{O}_X$-modules.

The variety $X$ determines a {\em simplicial conormal functor}
$\mathscr{C}_{{}_{\!X\!}}:\textsf{F}s\bold{S}\lra s\bold{Coh}_{X}$
as follows:

Each simplicial set $I_{\bullet}$ in $\textsf{F}s\bold{S}$
determines a cosimplicial $\Bbbk$-variety
$X^{I_{\bullet}}$
whose variety of $n$-simplices is the cartesian product
$X^{I_{n}}$.
Let $\Delta_{n}:X\xhookrightarrow{\ \ \ }X^{I_n}$
denote the diagonal.
Since these diagonals are "dual" to the terminal maps $I_{n}\lra\{\ast\}$,
they commute with every morphism
$X^{\phi}:X^{I_n}\lra X^{I_m}$
induced by by a morphism $\phi:I_{m}\lra I_n$ in $I_\bullet$.
For instance,
thinking purely in terms of coface and codegeneracy morphisms in
$X^{I_\bullet}$, this means that every triangle with nadir $X$ in the diagram
\begin{equation}\label{cosimp}
\xy
(0,0)*+{X^{I_0}}="X1";
(-20,0)*+{X^{I_1}}="X2";
(-40,0)*+{X^{I_2}}="X3";
(-60,0)*+{X^{I_3}}="X4";
(-67.5,0)*+{\cdots}="X5";
(-61,-2)*+{ }="X6";
(-64,-12.5)*+{\cdots}="X7";
(-30,-20)*+{X}="X";
{\ar@/_.5pc/ "X2"; "X1"};
"X1";"X2" **\crv{(-4,5.5) & (-16,5.5)} ?(.8)*\dir{>};
{\ar@/_.5pc/ "X1"; "X2"};
{\ar@/_.5pc/ "X3"; "X2"};
"X2";"X3" **\crv{(-24,5.5) & (-36,5.5)} ?(.8)*\dir{>};
"X2";"X3" **\crv{(-20,8) & (-40,8)} ?(.85)*\dir{>};
{\ar@/_.5pc/ "X2"; "X3"};
"X3";"X2" **\crv{(-36,-5.5) & (-24,-5.5)} ?(.8)*\dir{>};
{\ar@/_.5pc/ "X4"; "X3"};
"X3";"X4" **\crv{(-44,5.5) & (-56,5.5)} ?(.8)*\dir{>};
"X3";"X4" **\crv{(-41,8) & (-58,8)} ?(.85)*\dir{>};
"X3";"X4" **\crv{(-38,11) & (-60,11)} ?(.9)*\dir{>};
{\ar@/_.5pc/ "X3"; "X4"};
"X4";"X3" **\crv{(-56,-5.5) & (-44,-5.5)} ?(.8)*\dir{>};
"X4";"X3" **\crv{(-60,-8) & (-40,-8)} ?(.87)*\dir{>};
{\ar@/_1pc/_(.65){\Delta_1} "X"; "X1"};
{\ar@/_.5pc/_{\Delta_2} "X"; "X2"};
{\ar@/^.5pc/^{\Delta_3} "X"; "X3"};
{\ar@/^1pc/^(.7){\Delta_4} "X"; "X6"}
\endxy
\end{equation}
commutes.

At each $I_n$, we can form the conormal sheaf
$\mathscr{C}_{{}_{X}}(I_n)$
of the diagonal in $X^{I_n}$.
If $\mathscr{I}_{n}$ denotes the ideal sheaf of the diagonal in $X^{I_n}$,
then
$\mathscr{C}_{{}_{X}}(I_n)$ is, by definition, the $\mathscr{O}_{X}$-module
$$
\mathscr{C}_{{}_{X}}(I_n)=\Delta_{n}^{\ast}\mathscr{I}_{n}
$$
If $\phi:I_{m}\lra I_{n}$ is any morphism in $I_\bullet$,
let $f$ denote the morphism $f=X^{\phi}:X^{I_{n}}\lra X^{I_m}$
that $\phi$ induces in $X^{I_\bullet}$.
Since the triangle
$$
\xy
(0,0)*+{X}="X";
(-12,10)*+{X^{I_n}}="A";
(12,10)*+{X^{I_m}}="B";
{\ar@{->} ^{\Delta_{n}}"X"; "A"};
{\ar@{->}_{\Delta_{m}} "X"; "B"};
{\ar@{->}^{f} "A"; "B"}
\endxy
$$
commutes,
the morphism $f$ induces a morphism
$f^{\#}:f^{\ast}\mathscr{I}_{m}\lra\mathscr{I}_{n}$
of $\mathscr{O}_{X^{I_{n}}}$-modules,
which induces, in turn, a morphism
$$
\Delta_{n}^{\ast}(f^{\#})\ :\ \ \mathscr{C}_{{}_{X}}(I_m)\lra\mathscr{C}_{{}_{X}}(I_n)
$$
of $\mathscr{O}_{X}$-modules.
These morphisms $\Delta_{n}^{\ast}(f^{\#})$,
corresponding to all morphisms $\phi$ in $I_\bullet$,
equip the family $\big\{\mathscr{C}_{{}_{\!X\!}}(I_n)\big\}_{n=0}^{\infty}$
with the structure of a simplicial $\mathscr{O}_X$-module, that is, an object in
$s\bold{Coh}_{X}$:

$$
\xy
(0,0)*+{\mathscr{C}_{{}_{\!X\!}}(I_0)}="X1";
(-20,0)*+{\mathscr{C}_{{}_{\!X\!}}(I_1)}="X2";
(-40,0)*+{\mathscr{C}_{{}_{\!X\!}}(I_2)}="X3";
(-60,0)*+{\mathscr{C}_{{}_{\!X\!}}(I_3)}="X4";
(-70,0)*+{\cdots}="X5";
{\ar@/^.75pc/ "X2"; "X1"};
"X2";"X1" **\crv{(-16,6.5) & (-4,6.5)} ?(.8)*\dir{>};
{\ar@/^.75pc/ "X1"; "X2"};
{\ar@/^.75pc/ "X3"; "X2"};
"X3";"X2" **\crv{(-36,6.5) & (-24,6.5)} ?(.8)*\dir{>};
"X3";"X2" **\crv{(-40,9) & (-20,9)} ?(.85)*\dir{>};
{\ar@/^.75pc/ "X2"; "X3"};
"X2";"X3" **\crv{(-24,-6.5) & (-36,-6.5)} ?(.8)*\dir{>};
{\ar@/^.75pc/ "X4"; "X3"};
"X4";"X3" **\crv{(-56,6.5) & (-46,6.5)} ?(.8)*\dir{>};
"X4";"X3" **\crv{(-59,9) & (-43,9)} ?(.86)*\dir{>};
"X4";"X3" **\crv{(-62,11.5) & (-40,11.5)} ?(.89)*\dir{>};
{\ar@/^.75pc/ "X3"; "X4"};
"X3";"X4" **\crv{(-44,-6.5) & (-56,-6.5)} ?(.845)*\dir{>};
"X3";"X4" **\crv{(-40,-9) & (-60,-9)} ?(.88)*\dir{>};
\endxy
$$

The resulting simplicial conormal functor
$\mathscr{C}_{{}_{\!X\!}}:\textsf{F}s\bold{S}\lra s\bold{Coh}_{X}$
is right-acute.
Let $\mathrm{Ch}^{\leq0}\bold{Coh}_{X}$
denote the category of chain complexes of coherent
$\mathscr{O}_X$-modules, supported in degrees $\leq0$,
and let $\textsf{N}:s\bold{Coh}_{X}\lra\mathrm{Ch}^{\leq0}\bold{Coh}_{X}$
denote the {\em normalized chain complex} functor.
By Dold-Kan, $\textsf{N}$ is an equivalence of categories.
Hence the composite
$$
\textsf{N}\mathscr{C}_{{}_{\!X\!}}:\textsf{F}s\bold{S}\lra\mathrm{Ch}^{\geq0}\bold{Coh}_{X},
$$
which we will call the {\em conormal chains functor},
is right-acute.

By Theorem \ref{FMop}, the Fulton-MacPherson construction for
$\textsf{N}\mathscr{C}_{{}_{\!X\!}}$
produces an abstract operad
$$
\FM{\textsf{N}\mathscr{C}_{{}_{\!X\!}}}:\isop{(\textsf{F}s\bold{S})}\lra\Sets
$$
This operad is represented by an operad
$\mathrm{FM}_{\textsf{N}\mathscr{C}_{{}_{\!X\!}}}:\isop{(\textsf{F}s\bold{S})}\lra\bold{Sch}_{\Bbbk}$
in projective $\Bbbk$-schemes.
For each simplicial set $I_\bullet$ in $\textsf{F}s\bold{S}$,
we realize each variety
$\mathrm{FM}_{\textsf{N}\mathscr{C}_{{}_{\!X\!}}}(I_\bullet)$
as follows:

Let $\Qc(I_{\bullet})$
be the finite, acyclic quiver with vertices given by pairs $(J_{\bullet},n)$
consisting of a nonempty simplicial subset
$J_\bullet\sub I_\bullet$ and an integer $-\mathrm{dim}_{\ \!}I_\bullet\leq n\leq0$,
and with arrows in $\Qc(I_\bullet)$ of two forms, either
$$
(J_{\bullet},n)\lra(J_{\bullet},n-1),
\ \ \ \ \ \ \mathrm{or}\ \ \ \ \ \ 
(J'_{\bullet},n)\lra(J_{\bullet},n)
$$
for some proper inclusion $J'_{\bullet}\subset J_\bullet$
of nonempty simplicial subsets of $I_\bullet$.
The conormal chains functor $\textsf{N}\mathscr{C}_{{}_{\!X\!}}$
induces a representation $\qr{\textsf{N}C}_{{}_{\!X\!}}$
of $\Qc(I_\bullet)$ in $\bold{Coh}_{X}$,
with $\qr{\textsf{N}C}_{{}_{\!X\!}}(J_{\bullet},n)$
defined to be the $\mathscr{O}_X$-module
$\textsf{N}\mathscr{C}_{{}_{\!X\!}}(J_\bullet)_n$
of $(-n)$-chains in $\textsf{N}\mathscr{C}_{{}_{\!X\!}}(J_\bullet)$.
This representation takes each arrow of the form
$(J_{\bullet},n)\lra(J_{\bullet},n-1)$ in $\Qc(I_\bullet)$
to the differential
$\partial:\textsf{N}\mathscr{C}_{{}_{\!X\!}}(J_\bullet)_{n}\lra\textsf{N}\mathscr{C}_{{}_{\!X\!}}(J_\bullet)_{n-1}$
in the chain complex $\textsf{N}\mathscr{C}_{{}_{\!X\!}}(J_\bullet)$,
and takes each arrow of the form
$(J'_{\bullet},n)\lra(J_{\bullet},n)$ in $\Qc(I_\bullet)$,
for $J'_{\bullet}\subset J_\bullet$,
to the inclusion
$\textsf{N}\mathscr{C}_{{}_{\!X\!}}(J'_{\bullet})_{n}\xhookrightarrow{\ \ \ }\textsf{N}\mathscr{C}_{{}_{\!X\!}}(J_{\bullet})_{n}$
of $(-n)$-chains.
Thus a subrepresentation of
$\qr{\textsf{N}C}_{{}_{\!X\!}}$
is the same thing as a compatible family consisting of one subcomplex of each chain complex
$\textsf{N}\mathscr{C}_{{}_{\!X\!}}(J_\bullet)$,
for $J_{\bullet}\sub I_{\bullet}$.

There is a nonreduced {\em quiver quot scheme} $\mathrm{Quot}_{\Qc(I_\bullet)}(\qr{\textsf{N}C}_{{}_{\!X\!}})$
classifying such subrepresentations,
namely the closed subscheme in the product of classical quot schemes
$$
\mathrm{Quot}_{\Qc(I_\bullet)}(\qr{\textsf{N}C}_{{}_{\!X\!}})
\ \ \subset
\prod_{\O\neq J_{\bullet}\sub I_{\bullet}\ }\!\!\!\!\!\prod_{n=0}^{-\mathrm{dim}_{\ \!}I_\bullet}\!\!\!\mathrm{Quot}\big(\textsf{N}\mathscr{C}_{X}(J_n)/X/\Bbbk\big)
$$
cut out by the subrepresentation condition in
$\bold{Rep}_{\bold{Coh}/X}(\Qc(I_\bullet))$.
Since each of the classical quot schemes
$\mathrm{Quot}\big(\textsf{N}\mathscr{C}_{X}(J_n)/X/\Bbbk\big)$
already contains a connected component associated to each Hilbert polynomial, we can simply define
$$
\mathrm{FM}_{\textsf{N}\mathscr{C}_{{}_{\!X\!}}}(I_\bullet):=\mathrm{Quot}_{\Qc(I_\bullet)}(\qr{\textsf{N}C}_{{}_{\!X\!}})
$$
in the present example. There is no need to take anything like a disjoint union over dimension vectors.

These $\Bbbk$-schemes
$\mathrm{FM}_{\textsf{N}\mathscr{C}_{{}_{\!X\!}}}(I_\bullet)$
form an abstract operad
$\mathrm{FM}_{\textsf{N}\mathscr{C}_{{}_{\!X\!}}}:\isop{(\textsf{F}s\bold{S})}\lra\bold{Sch}_{\Bbbk}$
whose sets of $\Bbbk$-points recover the operad
$\FM{\textsf{N}\mathscr{C}_{{}_{\!X\!}}}:\isop{(\textsf{F}s\bold{S})}\lra\Sets$.

\end{example}

\vskip 1cm

Tyler Foster

\end{document}